\documentclass{amsart}
\usepackage[mathscr]{eucal}
\usepackage{amssymb, amsmath,array, amscd}
\usepackage{url}
\usepackage{multicol}

\usepackage[active]{srcltx}

\input xy
\xyoption{all}

\theoremstyle{change}
\newtheorem{thm}{Theorem}[section]
\newtheorem{THM}{Theorem}

\newtheorem{COR}[THM]{Corollary}

\newtheorem{prop}[thm]{Proposition}

\newtheorem{lemma}[thm]{Lemma}

\theoremstyle{definition}

\newtheorem{remark}[thm]{Remark}
\newtheorem{cor}[thm]{Corollary}

\newtheorem{example}[thm]{Example}
\newcommand{\hcF}{\hat{\mathcal{F}}}

\newcommand{\cF}{\mathcal{F}}
\newcommand{\cO}{\mathcal{O}}
\newcommand{\cG}{\mathcal{G}}
\newcommand{\cC}{\mathcal{C}}
\newcommand{\cU}{\mathcal{U}}
\newcommand{\cV}{\mathcal{V}}
\newcommand{\cW}{\mathcal{W}}

\newcommand{\Om}{\Omega}
\newcommand{\om}{\omega}
\newcommand{\la}{\lambda}

\newcommand{\C}{\mathbb{C}}

\newcommand{\N}{\mathbb{N}}
\newcommand{\Q}{\mathbb{Q}}
\newcommand{\R}{\mathbb{R}}
\newcommand{\Z}{\mathbb{Z}}
\newcommand{\PP}{\mathbb{P}}
\newcommand{\A}{\mathbb{A}}
\newcommand{\sym}{\mathrm{Sym}^k}

\newcommand{\f}{\phi}

\newcommand{\sing}{\mathrm{Sing}}
\renewcommand{\=}{:=}
\renewcommand{\div}{\mathrm{div}}
\newcommand{\ord}{\mathrm{ord}}
\newcommand{\dto}{\dashrightarrow}

\newcommand{\NS}{\mathrm{NS}}
\newcommand{\kod}{\mathrm{kod}}
\newcommand{\aut}{\mathrm{Aut}}

\newcommand{\GL}{\mathrm{GL}}

\begin{document}
\title{Webs invariant by rational maps on surfaces}
\author{Charles Favre}
\address{Charles Favre \\ CNRS and
Centre de Math\'ematiques Laurent Schwartz\\
 \'Ecole Polytechnique \\ 
91128 Palaiseau Cedex \\
France
}
\email{favre@math.polytechnique.fr}

\author{Jorge Vit\'orio Pereira}
 \address{IMPA, Estrada
 Dona  Castorina, 110\\
 22460-320, Rio de Janeiro, RJ, Brazil}
 \email{jvp@impa.br}

\thanks{The first author is supported by the ERC-starting grant project ”Nonarcomp” no.307856, and by the project "Ci\^encia sem fronteiras" founded by the CNPq}
 \date{\today}

\begin{abstract}
We  prove that  under mild hypothesis rational  maps on a surface preserving webs 
are of Latt\`es type. We classify endomorphisms of $\PP^2$ preserving webs, extending former results of Dabija-Jonsson.
\end{abstract}

\maketitle

\tableofcontents


\section{Introduction}
This paper is a sequel to~\cite{FP} where we classified non invertible  rational  maps on surfaces preserving a singular holomorphic foliation.
Here we extend our results to rational  maps on surfaces preserving a $k$-web, $k\ge2$. We also give a list
of all endomorphisms of $\PP^2$ preserving a web, thereby precising former results of
Dabija-Jonsson~\cite{DJ}. Dabija-Jonsson~\cite{DJ2} also  constructed many explicit examples of
endomorphisms preserving webs with all leaves algebraic.

\smallskip

\subsection{Terminology}
Before stating our main results, let us introduce some terminology.
Suppose $X$ is a smooth complex projective surface, and consider a dominant rational self-map $\f : X \dto X$.
Let $e(\f) \ge1$ be the topological degree of $\f$, and $\la(\f)\ge 1$ be its (first) dynamical degree.
By definition, $\la(\f)$ is equal to  the limit of $\| (\f^n)^* \|^{1/n}$ as $n\to \infty$ where $\f^*$ denotes the linear action
of $\f$ on the real Neron-Severi space $\NS_\R(X)$, and $\| \cdot \|$ is any fixed norm on $\NS_\R(X)$.
These degrees are birational invariants, and satisfy $ e(\f) \le \la(\f)^2$.
In the case of a holomorphic map of $\PP^2$, one has $ e(\f) = d^2$, $\la(\f) =d$ where $d$ is the degree of a
preimage of a generic line by $\f$.
We shall call the non-negative real number $$h(\f) := \log \max \{ e(\f), \lambda(\f)\}$$  the \emph{algebraic entropy} of $\f$.
It measures the volume growth of images of subvarieties in $X$ under iteration, and is an upper bound for the topological entropy of $f$ see~\cite{entropy,bornesup}. When $h(\f) =0$ then $\f$ is birational and either preserves a rational or elliptic fibration, or is the one-time map of a holomorphic vector field, see~\cite{Gsurface,DF}.

A rational map on a $d$-dimensional variety is called \emph{Latt\`es-like} if there exists a $\f$-invariant Zariski open   subset $\cU \subset X$
such that $\cU$ is a quotient of $\C^d$ by a discrete subgroup of affine transformations  acting properly on $\C^2$ and $\f$ lifts to an affine map on $\C^d$.
Note that $\cU$ inherits a natural affine structure from the standard one on $\C^d$, and this structure is preserved by $\f$. Latt\`es-like maps appear in the context of special rational maps  whose structure are reminiscent of the integrability of hamiltonian systems: endomorphisms of $\mathbb{P}^d$ having a non-trivial centralizer are Latt\`es-like by~\cite{Di01,DN02}; rational surface maps preserving a foliation without first integral are also Latt\`es-like by~\cite{CF,FP}.

A  $k$-web $\cW$ on a surface $X$ is, roughly speaking, a (singular holomorphic) multi-foliation with $k$ distinct local leaves  through a generic
point of $X$. At a generic point of $X$, the web $\cW$  is the superposition of  $k$ pairwise transverse  foliations $\cF_1, \ldots \cF_k$, and
one writes  $\cW = \cF_1 \boxtimes \ldots\boxtimes \cF_k$. A web is said to be \emph{irreducible} when
it cannot be (globally) decomposed  as a union of two proper sub-webs.
A $k$-web is \emph{parallelizable} if there exists a point $p$ and coordinates near $p$
such that $\cW$ is given by $k$  foliations by parallel lines in some coordinate system.
We refer to Section~\ref{S:basicwebs} for more details on these notions.

\subsection{Main results}

Our first results deal with  $k$-webs with $k\ge3$ invariant by  rational maps having positive algebraic entropy.  More precisely, we prove the following
\begin{THM}\label{TI:A}
Suppose $\f: X \dto X$ is a dominant rational map satisfying $h(\f)>0$ and preserving a $k$-web $\cW$ with $k\ge3$.
Then $\f$ is Latt\`es-like and $\cW$ is a parallelizable web. Moreover $\f$ preserves a one parameter family of webs.
\end{THM}

Our result is actually more precise, as there are some restrictions on a Latt\`es-like map to preserve
a $k$-web, $k\ge3$.

\smallskip

The case of $2$-web is quite specific since any map preserving a foliation with no first integral preserves a second foliation by~\cite[Corollary~B]{FP}.
In face of this result it is natural to restrict our attention to the case of irreducible $2$-webs.

Choose any rational map $\psi : \PP^1 \to \PP^1$. The product map $(\psi, \psi)$ on $\PP^1 \times \PP^1$ induces an endomorphism on $\PP^2 \simeq (\PP^1 \times \PP^1) /\langle  (x,y) \sim (y,x) \rangle$ that preserves the $2$-web whose leaves are image of fibers of (any of) the projection map $\PP^1 \times \PP^1 \to \PP^1$. This class of maps was introduced by Ueda in \cite{Ueda95}, and we shall refer to them as Ueda maps in the sequel.

\begin{THM}\label{TI:B}
Suppose $\f: X \dto X$ is a dominant rational map satisfying $h(\f)>0$ and preserving an irreducible  $2$-web $\cW$.
If $\f$ does not preserve any other foliation or web then $X$ is a rational surface, all  leaves of $\mathcal W$ are
rational curves, and $\f$ is semi-conjugated (by a rational map) to a Ueda map induced from a map on $\PP^1$ that is not conjugated to a finite quotient of an affine map in the sense of \cite{milnor}.
\end{THM}
An endomorphism of the Riemann sphere is a finite quotient of an affine map
if it is either a monomial, or a Chebyshev or a Latt\`es map.
\smallskip

Combining the previous results with our classification of rational surface maps preserving a foliation
presented in~\cite{FP}, we classify all endomorphisms of $\PP^2$ preserving a $k$-web. We present our results in two separate statements, one concerning the case $k\ge3$, and the other one concerning $2$-webs.

\begin{THM}\label{TI:C}
Suppose $\f: \PP^2 \to \PP^2$ is a holomorphic map of degree at least $2$, preserving a $k$-web $\cW$ with $k\ge 3$.
\begin{enumerate}
\item
Suppose $\f$ does not admit any totally invariant curve.

Then there exists a complex $2$-torus $X$, a collection of linear foliations $\cF_1, \ldots \cF_l$ on $X$, a complex number $\zeta \in \C^*$, and a finite group $G$ of  automorphisms of $X$ such that $X/G$ is isomorphic to $\PP^2$, and the triple
$(\PP^2,\f^N,\cW)$ is the push-forward under the natural quotient map  of
$(X, \zeta \times \mathrm{id}, \cF_1 \boxtimes \cdots \boxtimes \cF_l)$ with $N \in \{1,2,3,4,5,6\}$.

\item
Suppose $\f$ admits a totally invariant curve.

Then there exists a projective toric surface $X$, and a finite group $G$ of  automorphisms of $X$, such that the quotient space $X /G$ is isomorphic to $\PP^2$. The image
of $(\C^*)^2$ in $\PP^2$ is a Zariski open subset $U$ which is totally invariant by $\f$.
 Moreover,  one can find an integer $ d\ge 2$, and
a collection of complex numbers $\la_1, \ldots, \la_l$ such that the triple
$(U,\f^N,\cW)$ is the push-forward under the natural quotient map  of
$$\left((\C^*)^2, (x^d, y^d), \left( \left[ \la_1 \frac{dx}{x} + \frac{dy}{y}\right] \boxtimes \cdots \boxtimes \left[ \frac{dx}{x} + \la_l \frac{dy}{y} \right] \right)\right)~,$$ with $N \in \{ 1,2,3\}$.
\end{enumerate}
\end{THM}

We also give the complete list of groups that may appear in the two cases of the Theorem~C above.
As a direct consequence of this list, we obtain the following corollary.
\begin{COR}\label{CI:D}
Suppose $f: \PP^2 \to \PP^2$ is a holomorphic map of degree at least $2$ preserving an \emph{irreducible} $k$-web.
Then $k \in \{1, 2, 3, 4, 6, 8, 12\}$. Moreover, if $f$ admits a totally invariant curve then $k \in \{ 1, 2, 3, 4, 6 \}$.
\end{COR}

Finally we have the following result on endomorphism of $\PP^2$ preserving a $2$-web.
Observe that the situation is substantially more flexible than in the previous cases. 

\begin{THM}\label{TI:E}
Suppose $\f: \PP^2 \to \PP^2$ is a holomorphic map of degree at least $2$, preserving a (possibly reducible) $2$-web $\cW$. If $\f$ does not preserve any other web then
the pair $(\cW, \f)$ is one of the following:
\begin{itemize}
\item[(i)] $\cW$ is the union of two pencil of lines, and in suitable homogeneous coordinates $\f(x,y) = (P(x), Q(y))$ or $(Q(y), P(x))$
where $P$ and $Q$ are polynomials of the same degree  which are  not  conjugated to  monomial maps nor to  Chebyshev maps;
\item[(ii)] $\cW$ is the algebraic web dual to a smooth conic and $\f$ is an Ueda map whose associated map on $\PP^1$ is not a finite quotient of an affine map;
\item[(iii)] $\cW$ is the union of the pencil of lines $\{x/y = \mathrm{cst}\}$ and the pencil of curves
$\{x^p y^q = \mathrm{cst}\}$, with $p,q\in \N^*$, $\gcd \{ p,q\} =1$ and $\f (x,y) = (x^d/R(x,y) , y^d/R(x,y))$ where $R(x,y) = \prod_{i=1}^l (1+ c_i x^p y^q)$, $l (p+q) \le d$, $c_i\in \C^*$ and the rational map
$\theta \mapsto \theta^d / (\prod_{i=1}^l (1+ c_i \theta) )^{p+q}$ is  conjugated neither to a monomial nor to a Chebyshev map;
\item[(iv)] $\cW$ is the union of the pencil of lines $x/y = \mathrm{cst}$ and the pencil of curves
$x y = \mathrm{cst}$, and $\f (x,y) = (y^d/R(x,y) , x^d/R(x,y))$ where $R(x,y) = \prod_{i=1}^l (1+ c_i x y)$, $2l \le d$, $c_i\in \C^*$ and the rational map
$\theta \mapsto \theta^d / (\prod_{i=1}^l (1+ c_i \theta) )^{2}$ is  conjugated neither to a monomial nor to a Chebyshev map; 
\item[(v)] $\cW$ is the quotient of  the foliation $[dx]$ on $\PP^1\times \PP^1$ by the action of $(\Z/2)^2$ generated by $(x,y) \mapsto (-x,-y)$  and $(x,y)\mapsto (y,x)$;
and $\f$ is the quotient of a polynomial endomorphism of $\PP^1\times \PP^1$ of the form $(f(x),f(y))$ where $f$ is an odd or even polynomial which is is  conjugated neither to a monomial nor to a Chebyshev map; 
\item[(vi)] $\cW$ is the quotient of the foliation $[dx]$ on $\PP^1 \times \PP^1$ by the action of $(\Z/2)^3$ generated by $(x,y)\mapsto (y,x)$, $(x,y)\mapsto (x^{-1},y^{-1})$,
and $(x,y)\mapsto (-x,-y)$; and $\f$ is the quotient of an endomorphism of $\PP^1 \times \PP^1$ of the form $(f(x),f(y))$ where $f$ is a rational map
that commutes with  the group generated by $x \mapsto -x$ and $x \mapsto x^{-1}$ which is not a finite quotient of an affine map.
\end{itemize}
\end{THM}

Let us conclude this brief introduction by observing  that the results of Dabija and Jonsson \cite{DJ2} put a different perspective on webs preserved by endomorphisms. Namely, the authors classified webs defined as the dual of a (possibly reducible) curve $C$ in $\check{\PP}^2$ that are preserved by a non-invertible endomorphism. Examples (i) and (ii)  from Theorem~\ref{TI:E} fall into their classification whereas examples (iii) to (vi) not. All other examples from their paper (when $C$ has degree  three) fall into the range of application of Theorem~\ref{TI:C}.

\subsection{Outline of the paper}

In Section \ref{S:basicwebs} we recall the basic properties of webs. 
Section \ref{S:effect} contains some key computations that will lead to the proof of our main theorems. 
We first explain how to reduce the classification of maps preserving a web to maps preserving a foliation (Proposition \ref{p:weblift}). Then we describe how local invariants of webs behave under pull-back by rational maps (Lemma \ref{L:deg-web}). As a consequence we obtain a bound on the degree of an invariant web by an endomorphism (Proposition \ref{p:deg-web}).

The proof of Theorem~A is given in Section \ref{S:thmA}. It relies heavily on our former classification of invariant foliation by rational maps given in \cite{FP}.
The proof of Theorem~B (resp. C and E) can then be found in Section \ref{S:thmB} (resp. \ref{S:thmC} and \ref{S:ProofE}).

We conclude by describing an explicit irreducible $12$-web on $\mathbb P^2$ preserved by an endomorphism
in Section \ref{S:exam}.


\section{Basics on webs}\label{S:basicwebs}

A (singular holomorphic) $k$-web $\cW$ on a smooth complex surface $X$ is determined by a section  $\omega \in H^0(X, \sym \Omega^1_X \otimes N\cW)$
for some line bundle $N\cW$ on $X$ subject to the following conditions
\begin{enumerate}
\item \label{i} if not empty the zero set of $\omega$ has codimension two;
\item \label{ii} for any point $p$ outside a proper closed hypersurface  in $X$ (after trivializing the line bundle $N\cW$) the $k$-symmetric $1$-form  $\omega(p) \in \sym \Omega^1_{X,p}$ can be written as the product
of $k$ pairwise distinct linear forms.
\end{enumerate}

Two sections $\omega, \omega'$ define the same $k$-web  if and only if  they differ by a global nowhere vanishing holomorphic function.

Concretely, in a suitable open covering $\{ U_i\}$ of $X$, $\cW$ is described in each $U_i$ by some $k$-symmetric $1$-form $\omega_i$ with isolated zeroes such that $\omega_i = g_{ij} \omega_j$ on $U_i \cap U_j$ for some non-vanishing holomorphic functions $g_{ij} \in \cO^*(U_i \cap U_j)$. The cocycle $\{ g_{ij}\}$ determines the line-bundle  $N\cW$, and the collection $\{\omega_i\}$ induces a global section $\omega \in H^0(X,\sym \Omega^1_X \otimes N\cW)$.

\subsection{Discriminant divisor}\label{S:discriminant}
The zero locus of the section $\omega$ is a finite set $\sing (\cW)$ called the singular locus of $\cW$. The points where condition (\ref{ii}) does not hold is a
 complex curve called the discriminant curve of the web.
Its irreducible components  have naturally attached multiplicities as they are described locally by the vanishing of the classical discriminant of homogeneous binary forms in two variables. 

To be more precise first  introduce local coordinates $(x,y)$ on  a sufficiently small open set $U\subset X$ in order to obtain an
natural identification of
$\Omega^1_X(U)$   with $\mathcal O(U) dx \oplus \mathcal O(U) dy$, and of  $\sym \Omega^1_X(U)$ with the homogeneous
polynomials of degree $k$ in the  ring $\mathcal O(U)[dx,dy]$.  Then the discriminant of an element
$$\omega = \sum_{i=0}^k a_{i} dx^i dy^{k-i} \in \sym \Omega^1_X(U)$$
is by definition
\[
\Delta(\omega) = \frac{R[ \omega, \partial_{dx} \omega] }{k^k a_k }  \, .
\]
where $\partial_{dx} \omega = \sum_{i=0}^k i a_i  dx^{i-1} dy^{k-i}$, and
$R[\omega, \partial_{dx}\omega]$ is the determinant of the corresponding  Sylvester matrix, see \cite[IV \S 8]{lang}.
Since the discriminant of binary forms over $\mathbb C$ obeys the following rules
\begin{equation}\label{E:TransDis}
\begin{array}{lcl}
\Delta(\lambda P) &=& \lambda^{2(\deg(P)-1)}\Delta(P) \, , \\
\Delta(g^* P) &=& \det(Dg)^{\deg(P){(\deg(P)-1)}} \Delta(P) \, .
\end{array}
\end{equation}
where  $\lambda \in \mathbb C^*$ and $g \in \GL(2,\mathbb C)$ is an invertible linear map, it follows
that the local discriminants $\Delta(\omega_i) \in \sym \Omega^1_X(U_i)$ patch together to a section  $\Delta(\omega)$ of $K_X^{ \otimes k(k-1)} \otimes N \cW ^{\otimes 2(k -1)}$.
The corresponding  divisor   is, by definition,  the discriminant divisor of
$\cW$ and will be denoted by $\Delta(\cW)$.

\begin{example}\label{ex:web in P2}\rm Let $\mathcal W$ be a $k$-web on the projective plane. Its degree is defined as the number of tangencies
between $\mathcal W$ and a generic line $\ell$. More precisely, if $\mathcal W$ is induced by a $k$-symmetric $1$-form
$\omega \in H^0 ( \mathbb P^2, \sym \Omega^1_{\mathbb P^2} \otimes N \mathcal W)$ and $\iota : \mathbb P^1 \to \mathbb P^2$ is
a generic linear embedding  then
\[
\deg(\mathcal W) = \deg ( \iota^* \omega )_0 \, ,
\]
where $\iota^* \omega$ is an element of $H^0 ( \mathbb P^1, \sym \Omega^1_{\mathbb P^1} \otimes \iota^* N \mathcal W )$. Since $\sym \Omega^1_{\mathbb P^1} = \mathcal O_{\mathbb P^1}(-2k)$, the identity $N\mathcal W = \mathcal O_{\mathbb P^2}(\deg(\mathcal W) + 2k)$ follows. Thus
\[
\deg (\Delta(\mathcal W) )  = -3 k (k-1) + (\deg(\mathcal W) + 2k) 2 (k-1) = (k-1) \left( 2\deg(\mathcal W ) + k  \right) \, .
\]
\end{example}

\begin{example}\rm
The discriminant divisor of a completely decomposable web  $\cW = \cF_1 \boxtimes \ldots\boxtimes \cF_k$ is
equal to twice the sum of the tangency divisors of pairs of foliations $\cF_i, \cF_j$  with $1\le i < j \le k$, that is
\[
\Delta(\cW) =   \sum_{1 \le i < j \le k} 2 \, \mathrm{tang}(\cF_i, \cF_j)   \, .
\]
In particular its support contains all the singularities of the foliations $\mathcal F_i$.
\end{example}

\subsection{Associated surface and monodromy}\label{S:monodromy}
Associated to a web $\mathcal W$ on a surface  $X$ is a (possibly singular, even non-normal) surface $\mathbb W$
contained in the  projectivization of the tangent bundle of $X$ which can be loosely defined as the graph of
web. More precisely, if $\mathcal W$ defined by a $k$-symmetric $1$-form $\omega$ as in the
previous section then
\[
\mathbb W  =  \{ (x,[v])  \in  \PP(T X) \, | \, \omega(x)(v) = 0  \} \, .
\]
The natural projection $\pi: \PP(TX) \to X$ when restricted to $\mathbb W$ is a ramified covering away from
the fibers over the singular set of $\mathcal W$ which is the set of points where the symmetric form $\omega$ is vanishing. The discriminant locus of this covering is precisely
the support of $\Delta(\mathcal W)$. If we set $\mathbb W^{0}$ as $\mathbb W \setminus \pi^{-1}(\Delta(\mathcal W))$ then
$\pi: \mathbb W^{0} \to X \setminus \Delta(\mathcal W)$ is an \'{e}tale covering of degree $k$.  The monodromy of this
covering, seen as a representation
\[
 \rho : \pi_1( X \setminus \Delta(\mathcal W),x ) \to  \mathrm{Perm}(\pi^{ -1}(x)) \cong \mathfrak S_k
\]
from the fundamental group of the complement of $\Delta(\mathcal W)$ to the permutation group of a fiber of $\pi$,
measures the
obstruction to globally decompose $\mathcal W$ as a product of global sub-webs.

\smallskip

A $k$-web $\mathcal W$ will be called irreducible if any of the following
equivalent conditions hold true:
\begin{itemize}
\item the image of the  monodromy $\rho$ of $\mathcal W$ is a transitive subgroup of $\mathfrak S_k$;
\item the associated surface $\mathbb W$ is irreducible;
\item the web $\mathcal W$ cannot be written as the superposition of two subwebs $\mathcal W_1 \boxtimes \mathcal W_2$.
\end{itemize}

Recall that on $\PP(TX)$  there is the so-called  contact distribution $\mathcal D$. It is the smooth codimension
one distribution which over a point $(x,[v]) \in \PP (TX)$ is generated by the tangents of the natural lifts
of curves on $X$  through $x$ with tangent direction $[v]$.

 Locally, over a neighborhood $U\subset X$ endowed with local coordinates
$(x,y)$, the distribution $\mathcal D$ is defined by the (projectivization of the)  $1$-form $\alpha= p dy - q dx$ where $p=dx$ and $q=dy$ are
thought of as linear coordinates on $T X_{|U}$.

The restriction of $\mathcal D$ to $\mathbb W$ is automatically integrable thanks to dimensional reasons. From the definition of
$\mathcal D$ it is clear that the leaves of the induced foliation, at least away from the pre-image of $\Delta(\mathcal W)$, correspond
to the natural  lifts of leaves of $\mathcal W$ to $\PP (TX)$. It follows that $\pi^* \mathcal W$, seen as a $k$-web on $\mathbb W$,
decomposes as the product  of the foliation  $\mathcal F_{\cW}$ determined by $\mathcal D$ and a $(k-1)$-web $\mathcal W'$, so that we can write ${\pi^*(\cW) = \mathcal F_{\mathcal W} \boxtimes \mathcal W'}$.

\subsection{Vocabulary from local web geometry} 
It is customary to say that a germ of  $k$-web on $(\mathbb C^2,0)$ is algebraic
if its leaves are lines parametrized  by a germ  of reduced algebraic curve in the dual projective space $\check{\mathbb P}^2$.

In our opinion, this is a particularly bad terminology  since all global webs on projective surfaces are defined by algebraic twisted differential
forms. Since it is hard to go against a common usage, we will  stick to it and will say that a global web $\mathcal W$ on a projective surface $X$ is
\emph{algebraic} if $X$ is the projective plane, and the leaves of $\mathcal W$ are lines in $\mathbb P^2$ parametrized by a reduced curve $C \subset \check{\mathbb P}^2$. We will say that $\mathcal W$ is \emph{algebraizable} if for a generic point  $p \in S$ the germ of $\mathcal W$ at $p$ is biholomorphic to
a germ of algebraic web.

A global $k$-web $\mathcal W$ will be called \emph{parallel} (resp. \emph{parallelizable})  if it is an algebraic (resp. algebraizable) web with dual curve equal to an union of concurrent lines.

\section{The effect of rational maps on webs}\label{S:effect}

\subsection{The Galois trick}Let $\mathcal W$ be an irreducible $k$-web on a projective surface $X$. As explained in  Section \ref{S:monodromy},
 if $\mathbb W \subset \mathbb P TX$ is the graph of $\cW$ and $\rho : \mathbb W \to X$ is the natural projection
    then $\rho^* \mathcal W$ decomposes as the product of a foliation $\mathcal F_{\cW}$ and a $(k-1)$-web $\cW'$.
Iteration of this  remark   allow us   to reduce the analysis of maps preserving webs to maps preserving
finite sets of  foliations.

\begin{remark}\label{rem:galois twist}
We shall make a small twist in the terminology and call {\it Galois} any non necessarily finite map 
$\pi : X \to Y$ such that there exists a finite group of automorphisms $G$ acting on $X$ and preserving the fibers of $\pi$ such that the induced map $X/G \to Y$ is a birational morphism.
\end{remark}

\begin{prop}\label{p:weblift}
Suppose $\f: X \dto X$ is a dominant rational  selfmap on a smooth connected surface $X$ preserving  an irreducible $k$-web $\cW$.
Then there exists a (possibly disconnected) smooth projective surface  $\hat{X}$, a surjective morphism
$\mu: \hat{X} \to X$, a finite group $G \subset \aut(\hat{X})$, $k$ foliations $\cF_1, \ldots , \cF_k$ on $\hat{X}$
and a dominant rational map $\hat{\f} : \hat{X} \dto \hat{X}$ such that:
\begin{itemize}
\item for all $g \in G$, one has $\mu \circ g = \mu$, and the natural map $\hat{X}/G \to X$ induced by $\mu$ is a birational
morphism (thus $\mu$ is  Galois);
\item $G$ preserves the $k$-tuple of foliations and acts transitively on it, so that $\mu^* \cW = \hcF_1 \boxtimes\ldots  \boxtimes \hcF_k$, and for all $i$, $\mu_* \hcF_i = \cW$;
\item $\mu \circ \hat{\f} = \f \circ \mu$, that is there exists a group morphism $\rho : G \to G$ such that $\hat{\f} \circ g = \rho(g) \circ \hat{\f}$ for all $g\in G$;
\item the map $\hat{\f}$ permutes the foliations $\hcF_i$.
\end{itemize}
\end{prop}
\begin{proof}
Let $X_1$ be the normalization  of $\mathbb W$ and $\mu_1= \rho$. As recalled above $\mu_1^* \cW$ decomposes as a sum $\cF_1 \boxtimes \cW_1$ of a foliation and a $(k-1)$-web. Observe that $\mu_* \cW' = \cW$ for any irreducible sub-web of $\mu^*_1 \cW$ (in particular when $\cW' = \cF_1$). Moreover $X_1$ is connected and $\f$ lifts in a natural way to $X_1$.

We may then repeat this operation and we obtain a sequence of $(k-i+1)$-to-$1$ finite cover
$\mu_i: X_i \to X_{i-1}$, foliations $\cF_i$ and webs $\cW_i$ such that $\mu_i^* \cW_{i-1} = \cF_i \boxtimes \cW_i$. The composition map $\mu \= \mu_1 \circ \ldots \circ \mu_k : \hat{X} \= X_k \to X$ is thus a $k!$-to-$1$ finite cover from a normal projective surface $\hat{X}$ to a smooth one $X$ such that  one can write
$\mu^* \cW = \hat{\cF}_1 \boxtimes \ldots \boxtimes \hat{\cF}_k$ for some foliations $\hat{\cF}_i$, $i = 1, \ldots ,k$.

Write $\mu^* \cW = \hat{\cF}_1 \boxtimes \ldots \boxtimes \hat{\cF}_k$ for some foliations $\hat{\cF}_i$, $i = 1, \ldots ,k$.

Choose an ordering of the leaves of the web passing through $\cW$, and denote them by $L_1, \ldots , L_k$.
Then there exists a unique point $q \in \mu^{-1}(p)$ such that the leaf of $\cF_i$ passing through $q$ is mapped by $\mu$ onto $L_i$. Indeed there is a unique $q_1\in \mu_1^{-1}(q)  \subset X_1$ at which $\cF_1$ is mapped by $\mu_1$ to $L_1$; and then a unique point $q_2 \in \mu_2^{-1}(q_1) \subset X_2$  at which $\cF_2$ is mapped by $\mu_1 \circ \mu_2$ to $L_2$, and a direct induction allows one to conclude.

From these observations we see that given any permutation $\sigma\in \mathfrak{S}_k$ of $\{ 1, \ldots , k\}$,
and given any point $q \in \hat{X}$ whose image by $\mu$ lies outside the discriminant of $\cW$, then there exists a unique point $g_\sigma(q) \in \mu^{-1}(\mu(q)) \in \hat{X}$
such that $\mu (\hat{\cF}_i,q) = \mu (\hat{\cF}_{\sigma(i)}, g_\sigma(q))$ for all $i = 1, \ldots, k$. 
 
In this way one obtains a holomorphic automorphism $g_\sigma$ of $\hat{X}\setminus \mu^{-1} (\Delta(\cW))$ such that $\mu \circ g_\sigma  = \mu$, and the group of automorphisms $\{ g_\sigma\}_{\sigma \in \mathfrak{S}_k}$
acts transitively on the fibers of $\mu$.

We now argue that any $g_\sigma$ extends as an automorphisms on $\hat{X}$. 
Fix $\sigma\in \mathfrak{S}_k$ and any $q$ lying in $\hat{X}$.  Since $\mu$ is a \emph{finite} map, there exists a small neighborhood $U$ containing $\mu(q)$ such that  $\mu^{-1}(U)$ is an open neighborhood of $\mu^{-1}(\mu(q))$ such that any of its connected component is included in a local chart with values in the bidisk. 
The restriction of $g_\sigma$ to the component $V$ of containing $q$ 
can be thus viewed as a holomorphic map defined on the complement of a curve
(namely $\mu^{-1}(\Delta(\cW))$) in $V$ with values in the bidisk, and thus extends holomorphically to $V$. It follows that $g_\sigma$ induces an automorphism of 
$\hat{X}\setminus\sing (\hat{X})$.
Since $\hat{X}$ is a normal surface, $g_\sigma$ extends to an automorphism of $\hat{X}$ as was to be shown. In particular, $\mu$ is a Galois cover.

Since any automorphism on $\hat{X}$ lifts to an automorphism on its minimal desingularization, we may also assume $\hat{X}$ is smooth.
The proposition then follows.
\end{proof}

\begin{remark}
Two remarks are in order. First $\hat{\f}$ may not fix the foliations in general.  Second it may happen that $\hat{X}$ is not connected. 
Nevertheless a suitable power of $\hat{\f}$ fixes a connected component of $\hat{X}$ and each foliation $\cF_i$. 
\end{remark}

\subsection{Behavior under rational pull-backs}
Let $\f : X \dashrightarrow Y$ be a dominant rational map between smooth projective surfaces. If $\mathcal W$ is a $k$-web on $X$ then
one can naturally define a $k$-web $\f^* \mathcal W$ on $X$ by pulling back a $k$-symmetric $1$-form $\omega \in H^0(Y,\sym \Omega^1_Y \otimes N\mathcal W)$
defining $\mathcal W$ by $\f$. 

If $\om$ is determined by a collection of symmetric $1$-form $\om_i$ in local charts $U_i$ satisfying $\om_i = g_{ij} \om_j$ so that the cocycle $\{g_{ij}\}$ determines $N \mathcal W$, 
then the pull-backs $\f^* \om_i$ satisfy the relations $\f^*\om_i = (g_{ij} \circ \f )  \f^*\om_j$.
Although the map $\f$ is not defined everywhere, it follows
the pull-back $\f^* \omega$ is a well-defined holomorphic section of $\sym \Omega^1_X \otimes \f^* N \mathcal W$. 

Be careful  that $\f^* \omega$ may
have codimension one irreducible components in its zero set. Thus the $k$-web $\f^* \mathcal W$ will be defined by the division of $\f^* \omega$
by a section of $\mathcal O_X( ( \f^* \omega)_0 )$. In other words, if  $C \subset X$ is an irreducible curve and one defines $m(C) = m ( \f, \mathcal W, C)$ as the vanishing order of $\f^* \omega$ along $C$ then
\[
N \f^* \mathcal W = \f ^* N  \mathcal W \otimes \mathcal O_X \left( - \sum m( C) \cdot C  \right)
\]
with the summation running through all the irreducible curves of $X$. Of course there is no need for an infinite
summation, since $\f^* \omega$ will vanish only on curves contained in the ramification divisor $\Delta_{\f}$ of $\f$.
Recall that  $\Delta_{\phi}$  is locally defined by the vanishing of $\det D\f$.
Its support is the critical set $\cC_\f$ of $\f$. And there exists divisor $K_X$ and $K_Y$ representing the canonical classes of $X$ and $Y$ respectively such that the equation
\begin{equation*}
 K_X = \f^* K_Y + \Delta_\f~,
\end{equation*}
holds (as divisors).

If $C \subset X$ is an irreducible curve then we say that $C$ is invariant by a web $\mathcal W$ if
the tangent space of $C$ at any of its smooth point is among the null directions of a $k$-symmetric $1$-form  $\omega$ defining $\mathcal W$. If $\omega$ is the  $k$-th power
of a $1$-form vanishing on $T_x C$ at any smooth point $x\in C$ then we say that the curve $C$ is completely invariant by $\mathcal W$.

\begin{lemma}\label{L:deg-web}
Let $\f: X \dashrightarrow Y$ be a dominant rational map and $\mathcal W$ be a $k$-web on $Y$. If $C\subset X$ is an irreducible curve and  $\f(C)$ is not a point
then the inequality $m( C) \le k \ord_C (\Delta_{\f})$ holds true. Moreover,
\begin{enumerate}
\item  $m( C) >0 $ if and only if $\f(C)$ is invariant by $\mathcal W$; and
\item\label{i:2}  if $m( C) = k \ord_C(\Delta_{\f})$  then $\f(C)$ is completely invariant by $\mathcal W$.
\end{enumerate}
\end{lemma}
\begin{proof}
Let $\div(\f^*\om)$ be the divisor of zeroes of the twisted symmetric form $\f^*\om$.
First note that its support is included in $\cC_\f$ the critical set of $\f$. Let $C$ be any irreducible component of $\div(\f^*\om)$. At a generic point $p\in C$, and after change of  coordinates both in the source and the target, the map $\f$ takes the form $(x,y) \mapsto (x^n,y)$ with $C = \{x=0\}$ and $n-  1 = \ord_C(\Delta_\f)$. We may also write the symmetric $k$-form as follows $\om = \sum_{i+j=k} a_{ij}(x,y) dx^idy^j$.

Since $x$ does not divide some $a_{i_0j_0}$, we have $\ord_x(\f^*\om) \le \ord_C(\Delta_{\f})\times i_0 \le \ord_C(\Delta_{\f})\times k$. The equality case holds only  when $x$ does not divide $a_{k0}$, and $\ord_x(a_{ij}) \ge (k-i) ( 1- \frac1n)$ for all $j \ge 1$. In particular, the equality $\ord_x(\f^*\om) = \ord_C(\Delta_{\f}) \times k$ implies $\f(C)$ is a component of the discriminant of $\cW$ completely invariant by $\cW$.
We have thus proved:
\begin{equation*}\label{e:div}
\div(\f^*\om) = \sum_{C\subset \Delta_{\f}} m(C)\, C \text{ and } m(C) \le k\times \ord_C(\Delta_\f)~,
\end{equation*}
and equality $m(C) = k\times \ord_C(\Delta_\f)$ implies $\f(C)$ is completely invariant by $\cW$.
\end{proof}

\begin{remark}\label{R:contachave}
In general the converse of item (\ref{i:2}) does not hold. In the notation of the above proof,  to ensure that the curve $\f(C)$ is  totally invariant it suffices to have $\ord_x(a_{ij}) \ge 1$ for every $j \ge 1$, while the condition for having $m(C) = k \ord_C(\Delta_\f)$ is stronger in general.
Let us describe in more details the situation in the case of a $2$-web, given by
$\om = a_0 dx^2 + a_1 dxdy + a_2 dy^2$ at a neighborhood of a general point of $\f(C)$.
Then the following assertions hold true.
\begin{enumerate}
\item If  $a_2\neq0$ on $\f(C)$   then $m(C) =0$; $\f(C)$ is not $\cW$-invariant;  and $C$ is not $\cW$-invariant.
\item If $\ord_x(a_2)\ge1$ and  $a_1 \neq0$ on $\f(C)$ then $m(C)  = \ord_C(\Delta_\f)$;   $\f(C)$ is $\cW$-invariant but not in the discriminant;  and $C$ is $\cW$-invariant but is  not contained
in the discriminant.
\item If  $\ord_x(a_2) = 1$, $\ord_x(a_1) \ge 1$  then $m(C) = 1+ \ord_C(\Delta_\f)$; $\f(C)$ is completely invariant; and  $C$ is not $\mathcal W$-invariant.
\item If    $\ord_x(a_2) \ge 2$, $\ord_x(a_1) \ge1$ then $m(C)  = 2\, \ord_C(\Delta_\f)$; and both $C$ and $\f(C)$ are completely $\cW$-invariant.
\end{enumerate}

In particular,  if  $\ord_C(\Delta_{\f})=1$,  $\ord_x(a_2) =1$, $\ord_x(a_1) \ge1$ then $m(C)  = 2\, \ord_C(\Delta_\f)$; and  $\f(C)$ is completely $\cW$-invariant, but $C$ is not $\cW$-invariant.
\end{remark}

\subsection{Bound on the degree of $k$-webs invariant by an endomorphism of $\mathbb P^2$}
With Lemma \ref{L:deg-web} at hand it is a simple matter to prove the following

\begin{prop}\label{p:deg-web}
Let $\mathcal W$ be a $k$-web on $\mathbb P^2$ that is left invariant by a non-invertible
 endomorphism $\f: \mathbb P^2 \to \mathbb P^2$ then $\deg(\mathcal W) \le k$.

Moreover $\deg(\mathcal W) = k$ iff $m(C) = k\ord_C(\Delta_\f)$ for all irreducible curves.
In particular, $\deg(\mathcal W) = k$ implies that any irreducible component of the set of critical values of $\f$ is included in the discriminant of $\mathcal W$.
\end{prop}
\begin{proof}
We apply the previous computation to the case $X = Y = \PP^2$. 
Choose a $k$-symmetric $1$-form $\omega \in H^0(Y,\sym \Omega^1_Y \otimes N\mathcal W)$ defining $\mathcal W$. 

Pick a generic line $\ell$, and denote by $i: \ell \to \PP^2$ the inclusion map. As in the Example~\ref{ex:web in P2} we have $\deg(\cW) = \deg \div(i^*\om) = \deg (\sym\Om^1_{\ell} \otimes i^*N\mathcal W) = -2k + \deg(N\mathcal W)$, hence $\deg(\cW) +2k = \deg(N\mathcal W)$.

Now observe that $\f^* \om$ is a section of $\sym \Omega^1_Y \otimes \f^* N\mathcal W$
whose  restriction to $l$ has degree
$\deg(\div (\f^*\om)) = (d-1) \deg(N\mathcal W)$ if $\f^* \cW = \cW$ where $d$ is the algebraic degree of $\f$.

Now Lemma~\ref{L:deg-web} implies  $\deg(\div (\f^*\om)) \le k \deg (\Delta_\f) = 3k (d-1)$, whence  $\deg(\cW) = \deg(N\mathcal{W}) -2k \le k$ as soon as $d\ge2$.

\smallskip

Suppose that  $\deg(\mathcal W) = k$. Then $\deg(\div (\f^*\om)) = 3k (d-1)$
hence $m(C) = k\ord_C(\Delta_\f)$ for any irreducible curve. In particular, 
the image of any irreducible component of the critical set lies in the discriminant of $\mathcal W$ by Lemma \ref{L:deg-web} (2).
\end{proof}

This proposition is the analogue of \cite[Proposition 2.2]{FP} which is the key computation leading to the classification
of foliations of $\mathbb P^2$ invariant by endomorphisms, see \cite[Theorem C]{FP}.


\section{Invariant $k$-webs with $k\ge3$}\label{S:thmA}


In this section we give a proof of Theorem~\ref{TI:A}.

\subsection{How many foliations an endomorphism can preserve?}

The key step in the proof of Theorem~\ref{TI:A} is to prove the following result.
Recall the definition of the algebraic entropy $h(\f)$ of a rational self-map from the introduction.

\begin{prop}\label{P:pencil-fol}
Let $\f : X \dashrightarrow X$ be a rational map.
Suppose $h(\f)>0$, and $\f$ preserves three distinct foliations $\cF_0, \cF_1, \cF_\infty$. Then the set of all $\f$-invariant foliations is a pencil parameterized by $\PP^1$. Moreover every $k$-web invariant by $\f$ is a product of $k$ of these foliations.
\end{prop}
\begin{proof}
We start by observing that space of all meromorphic $1$-forms over $X$ forms a $2$-dimensional vector space over the field of meromorphic functions $\C(X)$. Indeed pick any 
two meromophic $1$-forms $\om, \om'$ that are regular and independent at a point $p$.
It implies the two forms to be indepedent at a generic point, hence there is no non-trivial relation $\psi \om + \psi' \om' = 0$ for $\psi, \psi' \in \C(X)$.
Now for any other form $\om''$ not proportional to $\om'$, we have $\om '' \wedge ( \om - \psi \om') = 0$ where $ \psi = (\om \wedge \om'')/ (\om' \wedge \om'')$. Whence $\om'' = \psi' ( \om - \psi \om') $ for some $\psi' \in \C(X)$. This implies our claim.

\smallskip

Let us define each foliation $\cF_0, \cF_1, \cF_\infty$ by a global meromorphic $1$-form $\om_i$, $i\in \{ 0,1,\infty\}$.
By the above observation, we may multiply the $1$-forms in order to have $\om_0 - \om_1 + \om_\infty=0$.

 By $\f$-invariance, we may find meromorphic functions $h_i$ such that
$\f^*\om_i = h_i\, \om_i$ for all $i\in \{ 0,1,\infty\}$.
 From the equality
$$ h_1\, (\om_0 + \om_\infty) = h_1\, \om_1 = \f^*\om_1 = \f^*(\om_0+\om_\infty) = h_0\, \om_0 + h_\infty\,\om_\infty ~,$$
we conclude $h_1 = h_0 = h_\infty$. In the sequel, we shall write $h$ for this function.

\smallskip

Suppose now $\cG$ is an $\f$-invariant foliation which is determined by a meromorphic $1$-form $\om$.
Pick $a,b\in \C(X)\setminus \{ 0 \}$ such that $\om = a \om_0 + b \om_\infty$. The $\f$-invariance of $\cG$
imply the equality $a \circ \f / b\circ \f = a/b$. If the quotient map
$\mu \= a/b$ is non constant, then $\mu: X \dto \PP^1$ defines a rational fibration for which each fiber is preserved by $\f$.
Pick a generic fiber $F$ and look at the induced holomorphic map $\f_F: F \to F$.
Since the action of $\f$ on the base of the fibration is the identity, the topological degree of $\f$
equals $e(\f_F)$, and $\la(\f) = \max \{ e(\f), e(\f_F) \} = e(\f_F)$ also.
Since we assumed $\max \{ \la(\f), e(\f) \} >1$, the map $\f_F$ is non invertible and $F$ is either a rational or an elliptic curve. In any case, replacing $\f$ by an iterate if necessary, one can find a fixed point $p\in F$ which is repelling for $\f_F$. Since $F$ is an arbitrary fiber of $\mu$ and
the repelling periodic points of $\f_F$ are Zariski dense, one may further assume the three foliations $\cF_0, \cF_1$ and $\cF_\infty$ are smooth and pairwise transversal at this point. Since the three directions determined by these foliations are preserved by $\f$, the differential $d\f_p$ is a homothety, so that the action of $\f$ on the
 base of the fibration is not identity.
This contradiction shows that  any $\f$-invariant foliation is induced by a $1$-form $a\om_0 + b\om_\infty$ with $a,b \in \C$.

\smallskip

Let now $\mathcal W$ be an irreducible $k$-web that is invariant by $\f$, and distinct from  
$\cF_0$, $\cF_1$, and $\cF_\infty$.
As in Section \ref{S:monodromy} consider the normalization $\pi: \mathbb W \to X$ of the graph of $\mathcal W$, so that $\mathcal W = \cF \boxtimes \mathcal W '$ for some foliation $\cF$. 
Denote by $\cF_{0,W}$, $\cF_{1,W}$, and $\cF_{\infty,W}$ the lifts of  $\cF_0$, $\cF_1$, and $\cF_\infty$ respectively to $\mathbb W$. The preceding argument shows that 
$\cF$ can be defined by a meromorphic $1$-form $a \pi^* \om_0 + b \pi^* \om_1$ for some $a,b \in \C^*$. 
This implies that the  function sending a point $q \in \mathbb W$ to the cross-ratio of $T_q\mathcal F_{0,W},T_q\mathcal F_{1,W}, T_q\mathcal F_{\infty,W}$ and $T_q \cF$ viewed as elements of $\PP (T_p \mathbb W)$ is a constant.

Back to $X$ we see that for  a generic point $p\in X$  the cross-ratio of $T_p\mathcal F_0,T_p\mathcal F_1, T_p\mathcal F_{\infty}$ and the tangent space of any leaf of $\mathcal W$
through $p$ is the same. In other words, $\mathcal W$ is a foliation. 
\end{proof}

\begin{cor}\label{C:nofirstintegral}
If  $\f$ and $\cF_0, \cF_1, \cF_\infty$ are as in Proposition \ref{P:pencil-fol} then
$\f$ leaves invariant at least one foliation without rational first integral.
\end{cor}

\begin{proof}
Suppose this is not the case and let $f_i: X \to C_i$ be  rational first integrals for  $\cF_i$ with $i \in \{ 0, 1, \infty \} $.
Thanks to Stein's factorization theorem we can, and will, assume that the generic fiber of $f_i$ is irreducible.
Consequently there are rational maps $\f_i  : C_i \to C_i$ for which $f_i \circ \f = \f_i \circ f_i$, and  for any two
distinct elements $i,j \in \{ 0 , 1 , \infty \}$  the map $\f$ is semi-conjugate to $\f_i \times \f_j: C_i \times C_j \to C_i \times C_j$ through
$f_i \times f_j : X \dashrightarrow C_i \times C_j$.

We claim that at least two of the maps $\f_0, \f_1, \f_{\infty}$ are non-invertible. 
Otherwise $\f$ would be semi-conjugate to an automorphism 
This would imply $\f$ to have topological degree $1$, and dynamical degree $1$ by \cite{dinh-nguyen} so that $\f$ would have zero algebraic entropy. Thus we can assume  that  $\f_0$ and $\f_1$ are non-invertible. As such  the set of periodic points of both maps  is Zariski-dense, and the same is true for $\f$.

Take a periodic point $p$ of $\f$ outside the discriminant of $\cF_0 \boxtimes \cF_1 \boxtimes \cF_{\infty}$. For each of the foliations $(\mathcal F_{\la})_{\la \in \PP^1}$ given by Proposition \ref{P:pencil-fol} there is a unique leaf passing through $p$. Since $p$ is periodic this leaf is
sent to itself by some iterate of $\f$.

Suppose that for each $\la \in \PP^1$ the leaf of $\cF_{\la}$ passing through $p$ is algebraic, and let $C_{\la}$ be this leaf. 
Fix any very ample line bundle $L\to X$, and consider the Hilbert polynomial 
$P_{\la} (n) = \chi ( C_{\la} , L^{\otimes n})$. Since $\PP^1(\C)$ is uncountable, there exists a polynomial $P$ and an infinite set of $\la$'s such that 
$P_{\la} = P$. Let $Z$ be the Zariski closure  of the curves 
$C_{\la}$ such that $P_{\la} = P$ in the Hilbert scheme of curves on $X$ having $P$ as Hilbert polynomial.   Any curve in $Z$ is left fixed by $\f$ so that we may consider the (algebraic) web induced  in $X$ by the choice of any curve included in $Z$. It is an $\f$-invariant web all whose leaves pass through the point $p$. 
 Proposition \ref{P:pencil-fol} leads to a contradiction which establishes the corollary.
\end{proof}

\begin{remark}
For any integer $d\in \Z$, the monomial maps $(x,y) \mapsto (x^d, y^d)$ in $(\C^*)^2$
preserve infinitely many foliations each having a rational first integral. 
\end{remark}

\subsection{Proof of Theorem~\ref{TI:A}}

Start with a dominant rational map $\f : X \dto X$ with $h(\f) >0$ defined on a smooth (connected) surface. 

By  Proposition~\ref{p:weblift}, there exists a finite (Galois connected) cover $\mu: \hat{X} \to X$ such that some iterate $\f^N$  of $\f$ lifts  to $\hat{X}$ and preserves $k$ distinct foliations  $\hat{\cF}_1, \ldots , \hat{\cF}_k$.
Denote by $\hat{\f}_N$ this lift (a priori $\f$ does not lift to $\hat{X}$ since we assume it to be connected). 
Since $k\ge3$, we may apply Proposition~\ref{P:pencil-fol} so that $\hat{\f}_N$ preserves a pencil of foliations  $\hat{\cF}_\la$ parameterized by $\la \in \PP^1$.
Corollary \ref{C:nofirstintegral} implies at least one foliation $\hat{\cF}_{\la_*}$ has no first integral. 

\medskip

Suppose first that  the foliation $\hat{\cF}_{\lambda_*}$ has Kodaira dimension zero.  
We are in the condition to apply the classification of foliations of Kodaira dimension zero invariant by rational maps. The case of birational maps was treated in \cite[Theorem 1.1]{CF} and the case of non-invertible rational maps was treated in  ~\cite[Theorem~4.3]{FP}. Below we combine both results and describe foliations of Kodaira dimension zero invariant by rational maps of positive algebraic entropy.

\begin{thm}\label{T:clas-kod0}
	Let $(X,\mathcal F)$ be a foliation on a projective surface with $\kod(\cF) =0$  and without rational first integral. Suppose $\f$ is a dominant  rational map preserving $\cF$ and satisfying $h(\f)>0$.
	Then up to birational conjugacy and to a finite cyclic covering we are in
	one of the following five cases.
	\begin{enumerate}
		\item[$\kappa_0$(1):] The surface is a torus $X = \C^2/ \Lambda$, $\cF$ is a linear foliation and $\f$ is a linear diagonalizable map.
		
		\item[$\kappa_0$(2):] The surface is a ruled surface over an elliptic curve $\pi: X \to E$ whose monodromy
		is given by a representation $\rho: \pi_1(E) \to (\C^*,\times)$. In affine coordinates
		$x$ on $\C^*$ and $y$ on the universal covering  of  $E$,  $\cF$ is induced by the $1$-form $\omega =dy + \la x^{-1} d x$ where $\lambda \in \mathbb C$, and $\f(x,y) = (x^a,ay)$ where $a\in \Z \setminus \{ -1, 0, 1 \}$.
		
		\item[$\kappa_0$(3):] Same as in case~$\kappa_0$(2) with $\rho: \pi_1(E) \to (\C,+)$, $\omega= dy + \lambda dx$  where $\lambda \in \mathbb C$, $\f(x,y) = (\zeta x+b,\zeta y)$ where
		$b \in \mathbb C$, and $\zeta \in \C^*$ with
		$|\zeta| \neq 1$.
		
		\item[$\kappa_0$(4):] The surface is $\PP^1 \times \PP^1$, the foliation $\cF$ is given in affine coordinates by the form
		$\la x^{-1} dx + dy$ and $\f(x,y) = (x^a, ay)$ with $\la \in \C^*$ and $a \in \Z \setminus \{ -1, 0, 1\}$.
		
		\item[$\kappa_0$(5):] The surface is $\PP^1 \times \PP^1$, the foliation $\cF$ is given in affine coordinates by the form
		$\la x^{-1}dx + \mu y^{-1} dy$ and $\f(x,y) = (x^ay^b, x^cy^d)$ where $\la,\mu \in \C^*$,
		with $M= \left[\begin{smallmatrix}
		a &b \\c &d
		\end{smallmatrix}\right]
		\in \GL(2,\Z)$, with $\la/\mu \notin \Q_+$, $|ad-bc| \neq 0 $, and $M$ diagonalizable over $\C$.
		
	\end{enumerate}
\end{thm}

Remark that only in the cases $\kappa_0(1)$ and $\kappa_0(5)$ the map $\f$ can be birational.

After a birational conjugacy $\hat{Y} \dashrightarrow \hat{X}$ 
and a finite cyclic cover $\bar{Y} \to \hat{Y}$, we are in one of the five cases $\kappa_0(1)$ to $\kappa_0(5)$ of Theorem \ref{T:clas-kod0} . Denote by $\hat{\psi}_N$ and $\bar{\psi}_N$ the lifts of $\hat{\f}_N$ to  $\hat{Y}$ and $\bar{Y}$ respectively, and by $\bar{\cG}_\la$ the lift of $\hat{\mathcal F}_{\lambda}$ to $\bar{Y}$.

Assume we are in one of the three cases  $\kappa_0(2), \kappa_0(3)$ or $\kappa_0(4)$. Then $\bar{\psi}_N$ preserves a pencil of foliations (given by the $1$-form 
$dy + \la dx$ in the case  $\kappa_0(3)$, and by $dy + \la \frac{dx}x$ in cases
$\kappa_0(2), \kappa_0(4)$) hence this pencil coincides with $\bar{\cG}_\la$.
We can thus find a discrete subgroup $\Gamma$ of the affine group of $\C^2$ acting properly discontinuously such that: $\C^2/\Gamma$ is isomorphic to a Zariski-dense subset  $\cU \subset \hat{Y}$, any finite collection of foliation in the pencil $\hat{\cF}_\lambda$ gives rise to a parallel web, and $\hat{\psi}_N$ lifts to an affine map say $A_N$ on $\C^2$.

The same conclusion holds in the two remaining cases  $\kappa_0(1), \kappa_0(5)$ for the following reason. Let us discuss the case $\kappa_0(1)$. The other case is completely analogous.
The map $\bar{\psi}_N$ is an endomorphism of a $2$-torus. 
Since $h( \bar{\psi}_N) = h(\f^N)>0$,  $\bar{\psi}_N$ admits a periodic orbit. But $\bar{\psi}_N$ preserves at least three foliations, hence its differential at this periodic point is a homothety. It follows that $A_N$
preserves any foliation in the pencil  and that any finite collection of these foliations forms a parallel web.

\smallskip

We have proved that $\hat{\f}^N$ is a Latt\`es-like map.  We now argue  that $\f$ itself
is also Latt\`es-like. 

Recall that $\hat{X} \to X$ is Galois with Galois group $G$, and that $G$ acts by permuting the  foliations $\hat{\cF}_1, \ldots , \hat{\cF}_k$.  Pick any $g\in G$ and  denote by $g'$ the induced birational map on $\hat{Y}$.  Pick any point $p \in \C^2$
that does not lie in the critical set of $\pi$ such that $ g'$ is a local diffeomorphism at $\pi(p)$ and $g'(\pi(p))$ is not a  critical value of $\pi$. Denote by $\pi^{-1}$ any branch of $\pi$ defined in a neighborhood of 
$g' (\pi(p))$.  The map $ \pi^{-1} \circ \f \circ \pi$ then maps  a parallel $k$-web to a parallel $k$-web
 hence is affine by the next lemma.
\begin{lemma}\label{lem:affine}
Any local diffeomorphism $h : (\C^2,0) \to (\C^2,0)$ mapping a parallel $k$-web, $k \ge 3$, to a parallel $k$-web is affine.
\end{lemma}
\begin{proof}
Composing with a suitable affine maps at the source and at the target space, we may assume that $h$ preserves the three foliations $[dx]$, $[dy]$ and $[dx+ \la dy]$ for some $\la \neq 0$.
The preservation of $[dx]$ and $[dy]$ implies  that $h$ can be written under the form $h(x,y) = (g_1(x), g_2(y))$. Since $[dx+ \la dy]$ is fixed, we get $\la g_1'(x) = g_2'(y)$ hence $h$ is affine. 
\end{proof}
Denote by  $A_g$ this affine map. 
Since $\pi \circ A_g  = g' \circ \pi$ on some open set, it follows that any $g'$ lifts to an affine automorphism to $\C^2$. In particular, $G$ leaves the open subset $\mathcal U = \C^2/\Gamma'$ invariant. Since $G$ is finite we may find a birational map $\hat{Y}' \dashrightarrow \hat{Y}$ that is an isomorphism over $\mathcal U$ and such that any element of $G$ induces an automorphisms on $\hat{Y}'$. Denote by $G'$ the subgroup of automorphisms of $\hat{Y}'$ induced by $G$, and define $Y' = \hat{Y}'/G'$. Observe that  $Y'$ is birational to $X$ and contains a Zariski dense open subset ${\mathcal U} '$ isomorphic to the quotient space $\C^2/\Gamma'$ where $\Gamma'$ is the group of affine automorphisms generated by $\Gamma$ and the lifts of all $g'$ to $\C^2$. 

Now look at the map $\f'$ induced by $\f$ on $Y'$. It preserves a web whose lift to $\C^2$ consists of a parallel $k$-web, $k\ge3$. Using the same argument as before based on Lemma \ref{lem:affine} we conclude that $\f$ lifts to an affine map on $\C^2$.

The following diagram summarizes the principal steps of the above proof. The notation $/G$ means that the corresponding arrow is Galois with Galois group $G$, and $\sim$ indicates that the corresponding map is birational. 
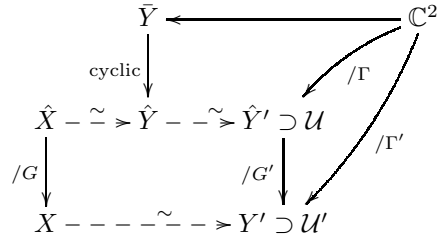
\[ \xymatrix{
&
\bar{Y}\ar[d]_{{\rm cyclic}}
&
&
\C^2
\ar[ll]
\ar@/_/[dl]^{/\Gamma}
\ar@/^/[ddl]^{/\Gamma'}
\\
\hat{X} \ar[d]_{/G} \ar@{-->}[r]^{\sim}
&
\hat{Y}\ar@{-->}[r]^{\sim}
&
\hat{Y}'\supset {\mathcal U} \ar[d]_{/G'}
&\\
X\ar@{-->}[rr]^{\sim}
&&
Y' \supset {\mathcal U}'
&
}
\]

\medskip

Suppose next that the  foliation $\hat{\cF}_{\lambda_*}$ has Kodaira dimension one. Since $\hat{\cF}_{\lambda_*}$ does not have a rational first integral, the rational map $\f^N$ is not birational according to \cite[Theorem 1.1]{CF}. The classification of non-invertible rational maps leaving invariant foliations of Kodaira dimension one without rational first integrals is carried out in \cite[Theorem 4.4]{FP}. There are two cases: $\kappa_1(1)$ and $\kappa_1(2)$.
In the case $\kappa_1(1)$, we are in the following situation. 
As above there is a cyclic cover $\bar{Y} \simeq \PP^1 \times \PP^1 \to \hat{Y}$ and a birational map
$\hat{X} \dashrightarrow \hat{Y}$ such that $\f^N$ lifts to $\bar{Y}$ as a map $\hat{\f}_N$ of the form
$(x,y) \mapsto ( ax ,x^my^k )$ with $a\in \C^*$, $m\in \Z$, $k \in \Z \setminus \{ -1, 0, 1\}$, 
$a^{n+1}= k$ and any foliation in the pencil  defined in $\bar{Y}$ by
$$
\om_{\la} = \frac{dy}y + \frac{m \, dx}{(k-1) x} + \la x^n dx
$$
is invariant by $\hat{\f}_N$.

Set $x= e^{z}$, $y = e^{w}$ and $ z'= e^{(n+1)z}$, $w'= w + \frac{m}{k-1} z$. 
Then $\hat{\cF}_\lambda$ is defined by $dw + \frac{m}{(k-1)} dz + \la e^{(n+1) z} dz$ in the $(z,w)$-plane and by the linear form $d w' + \frac{\la}{n+1} dz'$ in the $(z',w')$-plane. 
In particular, the web $\mathcal W$ is locally isomorphic to a pencil of linear foliations hence is parallelizable. 

We now prove that $\f$ lifts to $\bar{Y}$.
\begin{lemma}\label{lem:dyndeg}
Suppose $\f$ is a dominant rational map defined on a surface preserving an irreducible $k$-web $\mathcal W$  with $k\ge 2$ whose leaves are all algebraic. 
Then $e(\f) = \la(\f)^2$.
\end{lemma}
\begin{proof}
Since $\la(\f^N)= \la(\f)^N$ and $e(\f^N)= e(\f)^N$, up to replace $X$ by a birational model, we may assume that there exists a finite Galois cover $ \mu: \hat{X} \to X$ with group $G$, $\hat{X}$ connected, 
and a fibration $\pi: \hat{X}\to C$ such that the lift $\hat{\f}$ of $\f$ to $\hat{X}$
preserves $\pi$ in the sense that $h\circ \pi = \pi \circ \hat{\f}$ for some $h: C \to C$.
The image of the fibration under $\mu$ is the web $\mathcal W$ and the image of the fibration under the action of the group $G$ is the union of $k$ fibrations each preserved by $\hat{\f}$.
We then get a finite ramified map $\hat{X} \to C \times C$ that semi-conjugates $\hat{\f}$ to the map $(h,h)$. 

We conclude by observing that $\la((h,h)) = \la(\hat{\f})=\la(\f)$ and $e((h,h)) = e(\hat{\f})=e(\f)$, and
$e((h,h)) = \la((h,h))^2$, see \cite{dinh-nguyen}.
 \end{proof}
Observe that $\la(\f^N) = k = e(\f^N)$ hence $\la(\f) = e(\f)$.
Since $\hat{\cF}_0$ and $\hat{\cF}_\infty$ have rational first integrals and $h(\f)>0$
their images in $X$ are necessarily foliations by the previous lemma. 

Pick any generic point $p\in X$ and choose any local lift of the germ $\f : (X,p) \to (X, \f(p))$ to the $(z',w')$-plane. By Lemma \ref{lem:affine} this lift is affine, and since $\f$ preserves the two foliations   $\cF_0, \cF_\infty$ it is in fact diagonal. It follows that we can actually lift locally $\f$
to the $(z,w)$ plane as an affine map of the form $(z,w)\mapsto ( z + a, bw + c z + d)$
with $b\in \C^*$, $a,c, d\in \C$.
By the same argument as in the previous case, this implies $\f$ lifts globally to the $(z,w)$-plane hence is Latt\`es-like as was to be shown.

\medskip

In the case $\kappa_1(2)$ of \cite[Theorem 4.4]{FP}, the pencil of foliations is given in the product of the Riemann sphere with an elliptic curve $\bar{Y}=\PP^1 \times E$ by 
$$\om_{\la} = dy +  \la x^n dx ~,
$$
where $x\in \PP^1$ and $y\in \C$ lies in the universal cover of $E$.
The foliation $\hat{\cF}_\la$ is the pull-back of the linear foliation $dw + \frac{\la}{n+1} dz$ under the ramified cover $z = x^{n+1}$ and $w = y$ hence $\mathcal W$ is parallelizable.

There are only two foliations  with a rational first integral, namely $\hat{\cF}_0$ and $\hat{\cF}_\infty$, and they descend to foliations on $X$ by Lemma \ref{lem:dyndeg}.
Pick any generic point $p\in X$ and choose any local lift of the germ $\f : (X,p) \to (X, \f(p))$ to the $(z,w)$-plane. By Lemma~\ref{lem:affine} this lift is affine, and since $\f$ preserves the two foliations   $\hat{\cF}_0, \hat{\cF}_\infty$ it is in fact diagonal. It follows that we can actually lift locally $\f$ to $\bar{Y}$ as a linear map of the form $(x,y) \mapsto (ax, by)$ for some $a,b\in \C^*$. As before this implies that $\f$ lifts globally to $\bar{Y}$ hence is Latt\`es-like as was to be shown.

This concludes the proof of Theorem~\ref{TI:A}. 
\qed

\begin{remark}\label{rem:kod0}
Observe the proof gives the following statement. Suppose $\f: X \dto X$ is a rational map 
with $h(\f)>0$ that preserves a $k$-web with $k\ge 3$ that is the image under a finite ramified cover of a foliation without first integral and Kodaira dimension $0$.

Up to replace $X$ by a suitable birational model, there exists a Zariski open dense subset $\mathcal U$ and a discrete group $\Gamma$ of affine automorphisms of $\C^2$ such that
$\C^2/G$ is isomorphic to $\mathcal U$, $\phi$ lifts to an affine automorphism to $\C^2$ and $\mathcal W$ is the image of a linear foliation on $\C^2$.
\end{remark}


\section{Invariant $2$-webs}\label{S:thmB}
In this section, we prove Theorem~\ref{TI:B}.

Let $\mathbb W$ be the graph of $\mathcal W$ in $\PP(TX)$. The natural map $\pi : \mathbb W \to X$ has degree two. Let $\hat{X}$ be  the minimal desingularization of $\mathbb W$ so that the natural projection map $\mu: \hat{X} \to X$ is  Galois.

As in Proposition~\ref{p:weblift}, we get the existence of a rational map $\hat \f : \hat X \dashrightarrow \hat X$ satisfying $\mu \circ  \hat \f = \f \circ \mu$;
and of an involution $\iota : \hat{X} \to \hat{X}$ such that
\[
\mu^* \mathcal W = \mathcal F \boxtimes \iota^* \mathcal F \, .
\]
In particular,  $\mathcal F$ is isomorphic to $\iota^* \mathcal F$, and $\hat \f$  preserves the pair of foliations
$\mathcal F$ and $\iota^* \mathcal F$.   Moreover one of the following identities holds true:
\begin{equation}\label{E:iota}
\hat \f \circ \iota = \hat \f   \text{ or }  \hat \f \circ \iota = \iota \circ \hat \f.
\end{equation}

If $\mathcal F$ does not have a  first integral, then the same the same holds true for $\iota^* \mathcal F$. If no other foliation or web is invariant
by $\hat \f ^2$ then, according to the classification results of \cite{FP} and \cite{CF} and  up to birational maps and ramified coverings,
$\hat X$ is birationally equivalent to an abelian surface or to $\mathbb P^1 \times \mathbb P^1$ and a suitable power of $\hat \f$ is induced by a linear map in the first case, or a monomial map
in the second case. In both cases the map is induced by a diagonalizable matrix with distinct eigenvalues. But this fact is incompatible with the equations
(\ref{E:iota}). 

To see this, note that $\hat \f$ has a Zariski dense subset of periodic points. 
Pick any one of them $p$ and suppose it is fixed by $\hat \f$. Then the differential 
$d \hat \f (p)$ has two different eigenvalues whose eigenvectors are tangent to $\cF$ and $\iota^* \cF$ respectively. On the tangent space at $p$, the differential of $\iota$ is conjugated to the permutation $(x,y) \mapsto (y,x)$ of negative determinant hence
$d(\hat \f) (p) \circ d\iota (p)  \neq  d(\hat \f)(p)$, and $\hat \f \circ \iota \neq \hat \f$. 
But $d\iota (p)$ permutes the two eigenvectors of $d \hat \f (p)$ hence 
$\hat \f \circ \iota = \iota \circ \hat \f$ would imply the two eigenvalues of $d \hat \f (p)$
to be equal, a contradiction.
Thus both foliations have a rational first integral.

Assume   $f : \hat X \dashrightarrow C$ is a rational first integral for $\cF$. Without loss of generality assume the generic fiber
of $f$ is irreducible.   Since $\hat \f$ preserves $\mathcal F$, there exists a (regular) map $f_*\hat \f : C \to C$ such that
\[ f \circ \hat \f  = f_* \hat \f \circ f .
\]

The rational maps
$ F = f \times \iota^* f : \hat X \dashrightarrow C \times C $ and $f_* \hat \f \times f_* \hat \f $ fit into the commutative diagram
\begin{equation}\label{diag:1}
 \xymatrix{
\hat X \ar[dr] \ar@{-->}[dd]_(.65){F} \ar@{-->}[rr]^{\hat \f} && \hat X \ar[dr] \ar@{-->}[dd]^(.65){F}
\\
& X \ar@{-->}[dd] \ar@{-->}[rr]^(.4){\f} &&  X \ar@{-->}[dd]
  \\
 C \times C  \ar[dr] \ar[rr]^{f_* \hat \f \times f_* \hat \f  } && C \times C\ar[dr]  \\
&   C^{(2)}  \ar[rr]^{U_f} && C^{(2)} } \,
\end{equation}
where $C^{(2)}$ is quotient of $C^2$ by the natural involution, and the map from $C^{(2))}$ to itself is the  map $U_f$
induced by  $f_* \hat \f \times f_* \hat \f $. The frontal face of the cube shows that $\f$ is semi-conjugated  to the map $U_f$.

Since $h(\f)> 0$, the curve $C$ must be rational or elliptic. If it is elliptic then $f_* \hat \f \times f_* \hat \f $ would preserve a
whole $1$-parameter family of foliations contradicting our hypothesis.
Thus $C$ must be a rational curve hence $U_f$ is  a Ueda map.  

\smallskip

We now argue to prove that all leaves of $\mathcal W$ are rational. 
In fact we shall prove the stronger fact that all leaves of $\cF$ are rational. Observe that this implies 
$\hat{X}$ to be rational since it carries two independent rational fibrations, hence $X$  is 
rational too.

Observe first that $f_* \hat \f$ is a rational map on $\PP^1$ of degree $\ge 2$. 
Moreover it is not conjugated to a finite quotient of an affine map 
(see \cite{milnor} for a definition), since
otherwise the map $f_* \hat \f \times f_* \hat \f$ would preserve a pencil  of foliations
and $\f$ a non trivial web contradicting our assumption. 

Replacing $\f$ by a suitable iterate, we may suppose that $f_* \hat \f$ admits a fixed point
$p$ such that $D:= F^{-1} (\{p\} \times \PP^1)$ is an irreducible curve fixed by $\hat{\f}$ of (geometric) genus equal to the genus of a generic leaf of $\cF$. 
Now $F$ induces a semi-conjugacy from  $\hat{\f} : D \to D$ to 
$ f_* \hat \f : \PP^1 \to \PP^1$. It implies that $D$ is elliptic or rational. But the former case is impossible since otherwise $ f_* \hat \f$ would be a Latt\`es map (which is a finite quotient of an affine map). 
\qed

\begin{remark}
The $2$-webs described in the statement of Theorem \ref{TI:E}, see also Section \ref{S:ProofE}, show that it is not possible to replace the semi-conjugation by an actual conjugation in
Theorem \ref{TI:B}.
\end{remark}
\section{Endomorphisms of $\PP^2$ preserving a $k$-web with $k\ge3$}\label{S:thmC}
In this section, we prove Theorem~\ref{TI:C}.

Take $\f: \PP^2 \to \PP^2$ a holomorphic map of degree $d\ge 2$ preserving a $k$-web $\cW$, with $k\ge3$.
Apply Proposition~\ref{p:weblift}. We obtain a (possibly disconnected) surface $X_0$, a finite subgroup $G_0$ of $\aut(X)$, a rational map $\f_0$, and $k$ foliations $\cF_1, \ldots, \cF_k$, such that the quotient space $X_0/G_0$ is isomorphic to $\PP^2$, $\f_0$ descends to $\f$ on $X_0/G_0$ and $\cF_1\boxtimes \ldots \boxtimes \cF_k$ project to $\cW$. Notice, in general, that $\f_0$ is no longer holomorphic; that it does not preserve each foliation $\cF_i$ but only permute them; and that $\cW$ is not necessarily the projection of a single foliation.

\subsection{The classification of foliations invariant by rational maps}
Proposition~\ref{P:pencil-fol} implies some iterate $\f_0^N$ preserves a one-parameter family of foliations $\cF_\la$ containing the $\cF_i$'s and at least one foliation in this family does not admit any rational first integral. We may thus apply~\cite[Theorems~4.3,~4.4]{FP} (see also Theorem \ref{T:clas-kod0}). Note that the statement involves a birational conjugacy.
We can thus find a (smooth connected) surface $\hat{X}$, a rational map $\hat{\f}_N: \hat{X} \dto \hat{X}$, a foliation $\hat{\cF}$, and a finite group $\hat{G}$ of \emph{birational} maps on $\hat{X}$ such that:
$\hat{X}/\hat{G}$ is birational\footnote{that is $\hat{G}$ lifts as a finite subgroup  $\tilde{G}$ of $\aut( \tilde{X})$ for some birational model  $\tilde{X}$ of $\hat{X}$, and the surface
$\tilde{X}/\tilde{G}$ is rational.}
to $\PP^2$; $\hat{\f}_N\, \hat{G} = \hat{G}\, \hat{\f}_N$; the induced map by $\hat{\f}_N$ on the quotient space is $\f^N$; and $\hat{\f}_N$ preserves one of the one-parameter family of foliations $\hat{\cF}_\la$ listed in op.cit.

The topological degree and the dynamical degree are invariant under dominant morphisms, see~\cite{dinh-nguyen} (in the case of surfaces it is also a direct application of~\cite[Proposition~3.1]{BFJ}).
We hence have $e(\hat{\f}_N) = d^{2N}$ and $\la(\hat{\f}_N) = d^N$. Looking through the list of cases occuring in Theorem \ref{T:clas-kod0} (see also \cite[Theorems~4.3]{FP}) 
and \cite[Theorem 4.4]{FP} we see that this excludes all cases but $\kappa_0(1)$ and $\kappa_0(5)$.

\subsection{The case $\kappa_0(1)$: quotients of endomorphisms of abelian surfaces.}

Suppose first we are in case $\kappa_0(1)$. This time $\hat{X}$ is a $2$-torus, isomorphic to $\C^2/\Gamma$ for some lattice $\Gamma$; the map $\hat{\f}_N$ is linear associated to a matrix $M\in \GL(2,\C)$ preserving $\Gamma$; and $\hat{\cF}_\la$ is a one-parameter family of linear foliations preserved by $\hat{\f}_N$.
Since any birational map is regular on a $2$-torus, $G$ is a finite subgroup of automorphisms of $\hat{X}$.
As before, $\hat{\f}_N$ preserves a one-parameter family of foliations, hence $M$ is a homothety, and we can write $\hat{\f}_N = \zeta \times \mathrm{id}$ for some $\zeta \in\C^*$.
We shall make use of the following lemma.
\begin{lemma}\label{L:less-except}
Suppose $G$ is a finite subgroup of automorphisms of a $2$-torus $T$, and $\psi: T \to T$ is any holomorphic map on $T$. If there exists a birational map  $\mu : T/G \dto \PP^2$ conjugating the map induced by $\psi$ on $T/G$ to an endomorphism of the projective space of degree at least $2$, then $T/G$ is smooth, and $\mu$ is an isomorphism.
\end{lemma}
\begin{proof}[Proof of Lemma~\ref{L:less-except}]
Denote by $\psi_0$ the map induced by $\psi$ on $T/G$, and by $\phi := \mu \circ \psi_0 \circ \mu^{-1}$. The latter is an endomorphism of $\PP^2$ of degree $d\ge2$.

Note first that $T/G$ is a rational (normal) surface, and the field of meromorphic functions $\C(T)$ has transcendance degree $2$, hence $T$ is projective.

Suppose by contradiction that $\mu$ is not a regular map. Since $T/G$ is normal there exists a point $p$  whose total image under $\mu$ is a curve in $\PP^2$ by Zariski's main theorem, see e.g. \cite[Theorem V.5.2]{hartshorne}. In other words, 
we can  find a divisorial valuation $\nu_p$ on $\C(T/G)$ centered at $p$ whose image $\mu_* \nu_p$ is a divisorial valuation proportional to the order of vanishing $\ord_E$ along an irreducible curve $E \subset \PP^2$.

Observe that the torus map $\psi$ satisfies $e(\psi) = \la(\psi)^2$, hence cannot be factorized as a product map with one factor being an automorphism. In particular, the set of all preimages of $p$ by $\psi_0$ is Zariski dense in $T/G$. We may thus find a point $q \in T/G$ such that $\psi_0^M(q) =p$ for some positive integer $M$ and $\mu$ is a local biholomorphism at $q$. Since $\psi_0^M$ is a finite map, we can find a divisorial valuation $\nu_q$ on $\C(T/G)$ centered at $q$, such that $(\psi_0^M)_*\nu_q = \nu_p$. On the other hand $\mu_* \nu_q$ is centered at the point $\mu(q)$ in $\PP^2$, hence the valuation $\f^M_*\mu_* \nu_q$ is centered at the point $\f^M(\mu(q))$. This contradicts the equality $\f^M_*\mu_* \nu_q = \mu_* (\psi_0^M)_* \nu_q = \ord_E$.

Let us now prove that $\mu$ does not contract any curve. This will end the proof since
any birational regular finite map from a normal variety to another is necessarily an isomorphism.

We proceed by contradiction. Denote by $\mathcal C$ the (non-empty) set of points where $\mu$ is not locally finite, i.e. the union of all irreducible curves that are contracted to a point by $\mu$.  Since no curve in $T/G$  is totally invariant  by $\psi_0$ (otherwise its preimage in $T$ would also be), it follows that there exists an irreducible curve $C \subset T/G$ such that 
$\psi_0(C) \in \mathcal C$ but $C \not\subset \mathcal C$. The contradiction arises from the fact that  $\mu(\psi_0(C))$ is a point in $\PP^2$ equal to the irreducible curve $\phi (\mu (C))$. 
\end{proof}

It remains to see that $N\le 6$. To do so  we apply Theorem~\ref{TI:A}. Observe that we can apply Remark \ref{rem:kod0} so that $\f$ lifts to an affine map $A$ on $\C^2/\Gamma$ such that $A^N$ is a homothety.
We conclude by observing that the arguments of the eigenvalues of $A$ are roots of unity whose degree over $\Q$ is $1$, $2$ or $4$.

\subsection{The case $\kappa_0(5)$: quotients of monomial maps}
Suppose we are in case $\kappa_0(5)$. The situation is as follows: $\hat{X}$ is a smooth projective toric surface, $\hat{\phi}_N$ is a monomial map associated to a $2\times 2$ matrix $M$ with integral coefficients composed with a translation, and $G$ is a finite subgroup of
$\aut ((\C^*)^2 )$. Recall that the latter group is isomorphic the semi-direct product of $\GL(2,\Z)$ by $(\C^*)^2$.
There also exists a birational map $\mu : \hat{X}/ G \dto \PP^2$ conjugating the map induced by $\hat{\phi}_N$ on $(\mathbb C^*)^2/G$ to the iterate our endomorphism $\phi^N$ on the projective space.

From~\cite[Theorem~4.3]{FP}, we know that $M$ is diagonalizable. Both eigenvalues should be equal because $\hat{\f}_N$ preserves a
one-parameter family of foliations. In other words, in suitable coordinates $(x,y)$ (and replacing $\hat{\f}_N$ by a suitable power  if necessary) we have
$\hat{\f}_N(x,y) = (x^{d^N}, y^{d^N})$.

We now observe that Lemma~\ref{L:less-except} adapted to the present situation yields the following

\begin{lemma}\label{L:less-except2}
Suppose $G$ is a finite subgroup of $\aut((\mathbb C^*)^2)$, and $\psi: (\mathbb C^*)^2 \to (\mathbb C^*)^2$ is any holomorphic map on $(\mathbb C^*)^2$.
If there exists a birational map  $\mu: (\mathbb C^*)^2/G \dto \PP^2$ conjugating the map induced by $\psi$ on $(\mathbb C^*)^2/G$ to a non-invertible endomorphism of the projective space of degree at least $2$, then $\mu$ is an isomorphism from $(\mathbb C^*)^2/G$ onto its image. In particular the quotient space $(\mathbb C^*)^2/G$ is smooth.
\end{lemma}
\begin{proof}
Analogous to the proof of Lemma \ref{L:less-except}.
\end{proof}

Denote by  $U\subset \PP^2$ the image of $(\C^*)^2/G$ under $\mu$. This is a $\f$-totally invariant Zariski dense open subset of $\PP^2$. Its complement is a finite set of irreducible curves $L_1, \ldots , L_r$ (by \cite{CLN,SSU} we have $r \le 3$ and the $L_i$'s are lines but we shall not use this fact). 

Let us analyze the (rational) map $\pi : \hat{X} \dto \PP^2$ that semi-conjugate $\hat{\f}^N$ to
$\phi^N$. This map induces a regular proper finite map from $(\C^*)^2 = \pi^{-1}(U)$ onto $U$.

We claim that there exists a toric (possibly singular) variety $\bar{X}$ birational to $\hat{X}$ such that $G$ acts as a group of automorphisms on $\bar{X}$ and the new projection $\bar{\pi} : \bar{X} \to \PP^2$ is regular and finite. 

\smallskip

Granting this claim we conclude the proof. We first observe that any element $g \in G$ induces an automorphism on $\bar{X}$ since it permutes the fibers of $\bar{\pi}$ that is finite. 
It follows that $\PP^2 \simeq \bar{X} /G$.

Finally we need to argue that one can take $N \le 3$. 
We apply as above Remark \ref{rem:kod0}: we know that $\f$ lifts to an monomial map on $(\C^*)^2$ associated to a $2\times 2$ matrix $A$ such that $A^N = d^N\, \mathrm{id}$.
Now the arguments of the eigenvalues of $A$ are roots of unity of degree over $\Q$ less than  $2$ that is root of order  less than $3$. It follows that we may always take $N \le 3$.

\smallskip

Let us now prove our claim.  As in the proof of Lemma \ref{L:less-except} it is convenient to use the notion of divisorial valuations.

To each irreducible component $L$ of the complement of $U$ in $\PP^2$ corresponds a divisorial valuation $\ord_{L}$ defined on the field of rational functions of $\PP^2$.  The field  $\C (\hat{X})$ is a finite field extension of $\C( \PP^2)$, 
so that the set $\cV$ of all valuations on $\C (\hat{X})$ extending one of the divisorial valuations
$\ord_{L}$  is a finite set by \cite[Ch. VI \S 11 Theorem 19]{ZS}. This set is $G$-invariant and since the complement of $U$ is totally invariant by $\f$ is follows that $\cV$ is also totally invariant by the action of the monomial map $\hat{\f}^N$.
\begin{lemma}\label{L:pure-toric}
Suppose $\nu$ is a divisorial valuation on $\C(x,y)$ that is totally invariant by the monomial map $(x,y) \mapsto (x^d, y^d)$. Then $\nu$ is a monomial valuation.
\end{lemma}
In other words, each valuation $\nu$ lying in $\cV$  is determined by its value $(p_\nu,q_\nu)\in\Z^2$ on the monomials $(x,y)$ such that 
$$
\nu \left(\sum a_{kl} x^k y^l\right) = \min \{ k p_\nu + l q_\nu | \, a_{kl} \neq 0 \}~. 
$$
Recall that a toric surface is determined by a fan which is a collection of strictly convex rational cones in $\R^2$, see \cite{oda}. Consider the complete toric surface $\tilde{X}$ associated to the minimal $G$-invariant fan refining the one of $\hat{X}$ and including the rays $ \R_+\,(p_\nu,q_\nu)$ where $\nu$ ranges over all divisorial valuations in $\cV$. By construction a monomial divisorial valuation with weights $(p,q)$ on $x$ and $y$ has codimension $1$ center iff the ray $\R_+ (p,q)$ belongs to the fan of $\tilde{X}$.

Since the fan of $\tilde{X}$ is $G$-invariant, the group $G$ lifts as a  group of automorphisms of $\tilde{X}$, and we get a natural projection map $\tilde{X} \to \tilde{X}/G$ together with a birational map  $\tilde{\mu} :  \tilde{X}/G\dto \PP^2$.
We claim that the latter one is regular.

Let us shows that $\tilde{\mu}$ is regular. Since $\tilde{X}$ is normal the quotient space $\tilde{X}/G$ is also normal. If $\tilde{\mu}$ is not regular then we could find a divisorial valuation $\nu$ centered at a closed point $p\in 
 \tilde{X}/G$ whose image under $\tilde{\mu}$ is a divisorial valuation with codimension $1$ center in $\PP^2$. But $\tilde{\mu}$ is an isomorphism above $U$ hence 
$\tilde{\mu}_* \nu$ is proportional to $\ord_{L}$ for some irreducible compoenent $L$ of $\PP^2 \setminus U$. Pick any divisorial valuation $\tilde{\nu}$ on $\C( \tilde{X})$ that is mapped to $\nu$. Since the projection map $\tilde{X} \to \tilde{X}/G$ is finite, the center of $\tilde{\nu}$ is also a closed point in $\tilde{X}$. But $\tilde{\nu}$ belongs to $\cV$ hence
is a monomial valuation whose weights corresponds to a ray lying in the fan of $\tilde{X}$. 
But this means that the center of $\tilde{\nu}$ is a curve in $\tilde{X}$, a contradiction.

We have thus proved that the natural map  $\tilde{\pi} :  \tilde{X} \to \PP^2$ is regular. 
However in general the map $\tilde{\pi}$ is not finite. 
Consider the Stein factorization of $\tilde{\pi}$
\[
 \xymatrix{
\tilde{X} \ar[r]^{\varpi} \ar@/_1pc/[rr]_{\tilde{\pi}}
&
\bar{X} \ar[r]^{\bar{\pi}} 
&
\PP^2
} \,
\]
where $\varpi$ is a birational morphism, $\bar{\pi}$ is a finite map, and $\tilde{\pi} = \bar{\pi} \circ \varpi$. Consider the (finite) set of curves that are contracted by $\tilde{\pi}$ (this is the same as the ones contracted by $\varpi$). The key observation is that these curves are all torus-invariant because $\tilde{\pi}$ is a finite covering above $U$. It follows that $\bar{X}$ inherits  a unique structure of toric variety such that $\varpi$ is equivariant. 

This concludes the proof of our claim. The proof of Theorem~\ref{TI:C} is complete.
\qed

\begin{proof}[Proof of Lemma \ref{L:pure-toric}]
Pick any divisorial valuation $\nu$ on $\C(x,y)$. Write $p= \nu(x)$, $q = \nu(y)$. Since $\nu$ is divisorial, we may assume that $p$ and $q$ are coprime integers possibly replacing $\nu$ by a multiple.

Let $X$ be any smooth complete toric surface whose associated fan contains the ray $R= \R_+ (p,q)$, and let $E$ be the irreducible torus-invariant divisor associated to the ray $R$.
The valuation $\ord_E$ is the monomial valuation with weight $p$ on $x$ and $q$ on $y$.

Let us now consider the center $Z$ of $\nu$ in $X$. 
We claim that if $Z\neq E$ then it is a closed point lying on $E$ which is not torus-invariant. In other words $Z$ does not belong to any other irreducible torus-invariant curves.

Grant this claim. If $Z=E$, then $\nu$ is equal to $\ord_E$ and we are done. Otherwise $Z$ is a point which is not torus-invariant.  Observe that the map $(x^d,y^d)$ induces an endomorphism on $X$ that preserves the torus-invariant curve $E$ such that the torus-invariant points of $E$ are totally invariant by $f|_E$. 
In other words $E$ is a rational curve, and there exists a projective coordinates such that 
$f|_E(z) = z^d$ and $Z \neq 0,\infty$. It follows that $Z$ cannot be totally invariant, hence $\nu$ also, a contradiction. 

To complete the proof of the lemma, it suffices to justify our claim. 
Since $\nu$ is totally invariant, its center also and this implies $Z$ to be included in the union of all  torus-invariant curves. Pick any irreducible torus-invariant curve $E'$ and let $p'= \ord_{E'}(x)$, $q'= \ord_{E'}(y)$. When $E'\neq E$ the vectors $(p',q')$ and $(p,q)$ are not 
lying in the same ray, so that we can find $a,b \in \Z$ such that 
$p' a  + q' b > 0 > pa + qb$. It follows that $ \ord_{E'}(x^a y^b)> 0 > \nu (x^ay^b)$ hence
$Z$ is not included in $E'$. This proves our claim. 
\end{proof}

\section{Endomorphisms preserving a $2$-web}\label{S:ProofE}
In this section we prove Theorem~\ref{TI:E}.

Recall that our aim is to classify endomorphisms of $\PP^2$ preserving a $2$-web.
More precisely we fix an endomorphism $\f$ of $\PP^2$ of degree $d \ge2$ preserving 
a (possibly reducible) $2$-web $\mathcal W$ such that $\f$ does not preserve any other web. 

\begin{lemma}\label{L:degree1}
If $\cW$  is a $2$-web invariant by an endomorphism $\f : \PP^2 \to \PP^2$ of degree at least two
which does not preserve any other web then  the degree of $\cW$ is at most one.
\end{lemma}
\begin{proof}
	Let us first treat the reducible case. 
	When $\mathcal W = \cF_1 \boxtimes \cF_2$ is reducible, then we may apply \cite[Theorem~C]{FP} to both foliations $\cF_1$ and $\cF_2$. Since $\f$ does not preserve any other web we must fall into cases (1), (4) or (5) of op. cit. Observe that the two cases (4) and (5) are exclusive. It follows that both foliations have  degree $0$ or $1$ and one of them has degree $0$, whence  $\deg (\mathcal W) \le 1$ in the case. 
	
	In the sequel we shall assume that $\mathcal W$ is irreducible. It follows from Theorem \ref{TI:B} that all leaves of $\cW$ are rational curves.

	Let $e\ge1$ be the degree of a general leaf of $\mathcal W$, and denote by  
	$C \subset \mathbb P \left( H^0 \left(\PP^2, \mathcal O_{\PP^2}(e)\right)\right)$ the irreducible curve  parameterizing the leaves of the invariant web. 
	Denote by $Y \subset C\times \PP^2$ the incidence variety consisting of pairs $(D,p)$ such that 
	the point $p$ belongs to the curve $D$. The natural rational map $\mu :  C \times \PP^2 \dto \PP (T\PP^2)$
	sending  a pair $(D,p)$ to the pair $(p, T_pD)$ induces a 
	birational morphism from $Y$ to the graph $\mathbb W$ of the web in $\PP (T\PP^2)$. 
	
	Now denote by $\hat{\f}$ the map induced by $\f$ on $\mathbb W$; by $\psi$ the one induced on $Y$ and by $f$ the map induced on $C$ so that the following commutative diagram holds: 
	
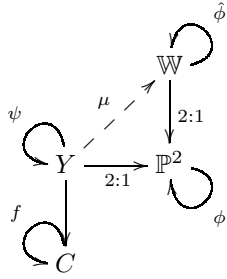
\[
	\xymatrix{
		& {\mathbb W} \ar@(r,u)[]_{\hat \f} \ar[d]^{2:1} 
		\\
		{Y} \ar@(u,l)[]_{\psi} \ar@{-->}[ur]^{\mu} \ar[d] \ar[r]_{2:1}
		&
		\PP^2 \ar@(r,d)[]^{\f}
		\\
		C\ar@(u,l)[]_{f}&} \,
	\]
	Let $\mathcal E$ denote the (finite) set of critical values of the fibration $Y \to C$.
	Pick any critical point $q\in C\setminus {\mathcal E}$ for $f$. The corresponding curve
	$D_q$ is then critical for $\psi$. Its image $\bar{D}_q$ in $\PP^2$ is thus a critical curve for $\f$ that is invariant by $\mathcal W$. Since $m(\bar{D}_q) = 2 \ord_{\bar{D}_q}(\Delta_\f)$ by Proposition \ref{p:deg-web} we are in Case (4) of Remark \ref{R:contachave}.
	It follows that $\bar{D}_q$ and $\f( \bar{D}_q)$ are totally invariant by $\mathcal W$.
	
	We have proved that the critical set of $f$ is included in the union of $\mathcal E$ together with the finite set of points $q$ for which $\bar{D}_q$ is invariant by $\mathcal W$.
	Denote by $\mathcal E'$ this union and observe that this finite set does not depend on $f$ but only on the fibration $Y \to C$ and the web $\mathcal W$.
	It follows that the critical set of all iterates of $f$ is included in $\mathcal E'$ hence
	$f$ is a monomial map. 
	
	This is the desired contradiction since otherwise $\f$ preserves a pencil of webs. 
\end{proof}

\begin{remark}
There are irreducible  $2$-webs of degree $2$ on $\mathbb P^2$ which are preserved by endomorphisms of degree greater than one. For instance,
the reducible $2$-webs on $\mathbb P^1 \times \mathbb P^1$ defined by $[ (dx/x + \lambda dy/y) ( dx/x + \lambda^{-1} dy/y )]$, $\lambda \in \mathbb C - \{ 0,1,-1\} $
are preserved by endomorphisms of the form $(x,y) \mapsto (x^d, y^d)$, and their quotient by the involution $(x,y)\mapsto (y,x)$ give rise to  examples
on $\mathbb P^2$.
\end{remark}

The previous lemma implies that $\mathcal W$ has degree $0$ or $1$.

Suppose first that the degree of $\mathcal W$ is $0$.
When $C$ is a smooth conic, then $\mathcal W$ is the image of (either one of) the two rational fibrations on $\PP^1\times \PP^1$ under the $2:1$ cover $\mu: \PP^1\times \PP^1\to \PP^2$ induced by the involution $(x,y)\mapsto (y,x)$, 
and $\PP^1\times \PP^1$ is isomorphic to the graph of $\mathcal W$ 
in  $\PP (T \PP^2)$. It follows that $\f$ lifts as an endomorphism to $\PP^1\times \PP^1$, and thus $\f$ is a Ueda example. 

\smallskip

When $C$ is reducible, then $C$ is a union of two lines and $\f$ preserves or permutes two pencils of lines in $\PP^2$. In suitable affine coordinates, then $\f(x,y) = (P(x), Q(y))$ or $(Q(y), P(x))$ for two polynomials of the same degree. Moreover these polynomials are neither conjugated to a monomial nor to a Chebyshev since otherwise $\f$ would preserve another web. 

\smallskip

Let us now assume that $\mathcal W$ has degree $1$. 
When $\mathcal W$ is reducible then we can apply \cite[Theorem~C]{FP}. Since at least one of the two foliations has degree $1$ we are in one of the exclusive case (4) or (5) of op. cit. 
It is then not difficult to conclude that we are in cases (iii) and (iv) of Theorem~\ref{TI:E} respectively. Observe that the rational map $\theta \mapsto \theta^d / (\prod_{i=1}^l (1+ c_i \theta) )^{p+q}$ appearing in case (iii) is the map induced on the base of the invariant rational fibration $\{x^py^q = \theta\}$ (in the notation of op. cit.). 
This rational map admits a totally invariant point hence is conjugated to a polynomial. When this polynomial is conjugated to a monomial or to a Chebyshev map then $\f$ preserves a web different from $\cW$ that is excluded from our assumption. Case (iv) is treated similarly.

\smallskip

Finally let us assume that $\deg (\mathcal W) =1$ and $\mathcal W$ is irreducible. 
We may then consider its dual foliation $\cF$ in  $\check{\PP}^2$ as defined in \cite[\S 3]{Marin}, see also  \cite[Section 1.4]{invitation}. This foliation is characterized by the fact that its leaf through the point $L\in \check{\PP}^2$ is tangent to the line associated  to the unique point $p\in L \subset \PP^2$ where $L$ and the web $\mathcal W$ admits a point of tangency. For a generic line this point is unique since the degree of $\mathcal W$ is one. A direct computation also shows  that the degree of $\mathcal F$ is $2$ because $\mathcal W$ is a $2$-web, see \cite{Marin}.
\begin{lemma}\label{L:discr}
Let $\cW$ be a $2$-web  on $\PP^2$. If $\cW$ is invariant by an endomor\-phism $\f: \PP^2 \to \PP^2$ of degree at least two then
its discriminant $\Delta(\cW)$ is invariant by $\cW$. 
\end{lemma}

\begin{proof}
	We proceed by contradiction. Pick  any irreducible component of the discriminant $C \subset \Delta(\cW)$ that is not invariant by $\f$.  Let $C_1$ by any irreducible component  of $\f^{-1}(C)$.
	Since $\f$ is finite we may find local coordinates $(x_1,y_1)$ at a generic point $p\in C_1$,
	and  local coordinates $(x,y)$ at its image $\f(p)\in C$ such that: 
	$C_1 = \{ y_1 =0\}$, $C= \{ y=0\}$, 
	$(x,y)= \f(x_1,y_1) = (x_1,y_1^k)$, and the $2$-web is given by $\cW = [ dx^2 + y^m dy^2]$ for some $m\ge 1$, see \cite[Lemma 2.1]{Marin}.
		
	It follows that  $\f^* \cW$ is defined at $p$ by the symmetric form $dx_1^2  + k^2 y_1^{km + 2k -2} dy_1^2$, hence the vanishing order of the discriminant of $\f^* \cW$ along $C_1$
	equals $k(m+2) -2 \ge 1$. 
	
	Summing up all these contributions over all components of $\f^{-1} (C)$ 
	we conclude that the degree of the discriminant is greater than $\deg(\f^* C) = \deg(\f) \deg(C)$. We get a contradiction by replacing $\f$ by an iterate $\f^n$ and let $n\to \infty$.\end{proof}

Let us see what this information implies for the dual foliation $\cF$.  To that end, recall that an inflection point $p$ of $\cF$ is a point  at which the line tangent to $\cF$ at $p$ is tangent to order at least $\ge 3$ to the leaf of $\cF$ passing through $p$. Since $\cF$ is not a pencil of lines, the set of inflection points is a curve.

Let $C$ be any irreducible component of the inflection curve of $\cF$.
At any point $p\in C\subset \check{\PP}^2$ where the foliation is smooth, we consider the line $L_p$ tangent to $\cF$ at $p$. Since the degree of $\cF$ is $2$, either $L_p$ is the leaf of $\cF$ through $p$, or $L_p$ is tangent to $\cF$ only at $p$ and the order of tangency is exactly equal to $3$. It implies that the  point corresponding  to $L_p$ in $\PP^2$ lies in the discriminant of $\cW$. Moreover the direction of the leaves of $\cW$ at $L_p$ are given by the line associated to $p$ in $\PP^2$.

Suppose that $C$  is not a line. This is equivalent to assume that $C$ is not a leaf of $\cF$. 
Then the set  $\{L_p\}_{p \in C}$ describes an irreducible component $D$ of the discriminant of $\cW$ in $\PP^2$ and the tangent direction of $\cW$ at $L_p$ is given by $p$. On the other hand the tangent direction of $D$ at $L_p$ is associated to the limit point $L_p \cap L_q \in \check{\PP^2}$ when $q \to p$ in $C$, and this limit point cannot be equal to $p$ for a generic choice of point on $C$.  Whence $D$ is not invariant by $\cW$.

Lemma \ref{L:discr} thus implies that every irreducible component of the inflection
curve of $\cF$  is a leaf of the foliation. In the terminology of \cite{Marin} this means that the foliation is  convex.

The classification of degree two  convex foliations defined over $\mathbb R$ was carried out in \cite{SV}.  From that classification we extract the following proposition. 

\begin{prop}
Let $\cF$ be a convex foliation of $\PP^2$ of degree $2$.
Then in suitable projective coordinates $[x:y:z]$  of $\PP^2$, there are only three possibilities:
\begin{enumerate}
\item in the affine chart $\{ z=1\}$  the foliation  is
defined by the $1$-form
\[
\omega = x(x-1) dy - y(y-1) dx ;
\]
\item in the affine chart $\{ z=1\}$  the foliation  is
defined by the $1$-form
\[
\omega = x^2 dy - y^2 dx ;
\]
\item in the affine chart $\{ z=1\}$  the foliation  is
defined by the $1$-form\[
\omega = x^2 dy - dx.
\]
\end{enumerate}
\end{prop}
\begin{proof}[Sketch of proof]
First observe that any degree $2$ foliations admits a curve of inflection of degree $6$, see e.g. \cite[\S 3.1]{Marin}. Since the inflection curve is invariant by the foliation, each of its component is a line. Counted with multiplicity it follows that $\cF$ is a degree $2$ foliation that admits exactly $6$ invariant lines. 

Pick any affine coordinates $x,y$ and any polynomial vector field 
$\chi=  p(x,y) \frac{\partial}{\partial x} +q(x,y) \frac{\partial}{\partial y}$ inducing $\cF$ where $p,q \in \C [x,y]$ are polynomials of degree $2$ without common factor.  

When $p$ and $q$ are real then $\chi$ belongs to $\mathrm{QSL}_6$
in the terminology of \cite[notation 3.1]{SV}, so that \cite[Theorem~3.1]{SV} applies. We now observe that the method used in op. cit. applies identically to complex vector fields.
Three different cases appear:
\begin{itemize}
\item
When $\cF$ has $6$ invariant lines each of multiplicity $1$ (this includes cases (1) to (4) of 
\cite[Table 1]{SV}), then $\chi$ can be taken to be equal to  $x(x-1) \frac{\partial}{\partial x} +y(y-1) \frac{\partial}{\partial y}$;
\item 
When $\cF$ has $2$ invariant lines of multiplicity $2$, and $2$ invariant lines of multiplicity $1$ (this includes cases (5) to (9) of 
\cite[Table 1]{SV}), then  $\chi$ can be taken to be equal to $x^2 \frac{\partial}{\partial x} +  y^2 \frac{\partial}{\partial y}$;
\item
When $\cF$ has $2$ invariant lines each of multiplicity $3$ (this includes cases (10) and (11) of 
\cite[Table 1]{SV}), then  $\chi$ can be taken to be equal to $x^2 \frac{\partial}{\partial x} + \frac{\partial}{\partial y}$.
\end{itemize}
Observe that the classification is actually simpler in the complex case since many real normal forms obtained by Schlomiuk and Vulpe are not conjugated over $\R$ but are conjugated over $\C$.  
\end{proof}

We finish the proof of Theorem~\ref{TI:E}  in the three lemmas below.

\begin{lemma}
If $\cW$ is the $2$-web dual to the foliation $[x^2 dy - dx]$ then $\cW$ is not preserved by any endomorphism $\f : \PP^2 \to \PP^2$
with $\deg(\f) \ge 2$.
\end{lemma}
\begin{proof}
In suitable affine coordinates the $2$-web $\cW$ is defined by the symmetric form
\[
 \omega  = x dy^2   - y (xdy - ydx)^2 = (x - yx^2 ) dy^2 + 2 xy^2 dx dy - y^3dx^2 \, .
\]
Its discriminant is therefore $ 4x^2 y^4 + 4 (x-yx^2)y^3 = 4 x y^3$. If $\f$ is an endomorphism preserving $\cW$ then the line $\ell = \{ y=0\}$ must be totally invariant 
and satisfy $\f^* \ell = \ell$ which contradicts our assumption that $\deg(\f)\ge2$.

Indeed, let $C$ be a curve such that $f(C) = \ell$,
and $\delta= \delta(C)$ be the vanishing order of the ramification divisor $\Delta_{\f}$ along $C$. We can choose coordinates
at a generic a point of $C$, keeping the coordinates at $\ell$, in such a way that $\f(x,y) = (x, y^{ \delta+1})$. This choice makes
evident that the vanishing order of $\f^* \omega$ along $C$ is exactly $2 \delta$, and that the discriminant of $\cW$ has order
at least $\delta + 3$ along $C$. This implies that $C= \ell$ so that $\ell$ is totally invariant by $\f$ and $\delta=0$ hence $\f^* \ell = \ell$.
\end{proof}

\begin{lemma}
Suppose $\f$  preserves the  $2$-web  $\cW$  dual to the foliation $[x^2 dy - y^2dx]$.
Then we are in case (v) of Theorem~\ref{TI:E}.
\end{lemma}
\begin{proof}
In suitable affine coordinates $\cW$ is defined by the symmetric form
$\omega = x dy^2 - y dx^2.$

Notice that the endomorphism 
$\pi(x,y) = ((x+y)^2, (x-y)^2)$ realizes the affine space $\A^2$ as the quotient of $\A^2$ by the
the order $4$ group generated by the involutions  $i_1 (x,y)= (y,x)$ and $i_2(x,y)=(-x,-y)$. Since
\[
\pi^* \omega = -8 (x+y)^2(x-y)^2 dx dy \, ,
\]
it follows that $\cW$ is the image of $[dx]$ by $\pi$.

Consider now the quadric cone $X:= \A^2 / \langle i_1 \rangle$ and the natural map
$\varpi: X \to \A^2$ induced by $\pi$ which is a $2:1$ cover.  Since $[dx]$ is preserved by $i_1$, it induces a foliation $\cF_0$ on $X$ which is mapped to $\cW$ under $\varpi$ but is not preserved by the involution determined by $\varpi$.

It follows that there is a canonical birational proper map from the closure of the graph of $\cW$ in the projectivized tangent space of $\A^2$ onto $X$. This map is actually an isomorphism
over the set of regular points $X_0$ of $X$ (the singular set $X\setminus X_0$ consists of a single point).  It follows that $\f$ lifts to an endomorphism of $X$ whose second iterate preserves $\cF_0$.

Now the map $\A^2  \setminus \{ (0,0)\} \to X\setminus X_0$ is a universal cover hence
$\f$ lifts to $\A^2 \setminus \{ (0,0)\}$ and thus to $\A^2$ as a polynomial endomorphism
$\hat{\f}$ that preserves the foliation $[dx]$ and commutes with the involution $i_1$ and with the group   generated by $i_1$ and $i_2$.

It is then not difficult to see that one can assume that the lift has the following form
$\hat{\f} (x,y) = (f(x), f(y))$ with $f$ an odd (or even) polynomial. 
\end{proof}

\begin{lemma}
Suppose $\f$  preserves the  $2$-web  $\cW$  dual to the foliation $\cF_2=[x(x-1) dy - y(y-1)dx]$.
Then we are in case (vi) of Theorem~\ref{TI:E}.
\end{lemma}
\begin{proof}
The $2$-web is given in affine coordinates by
$$
\om = 
x (x-1) dx^2 + (2xy) dx dy - y (y-1) dy^2~.
$$
Let $G$ be the group generated by the three involutions  $i_1 (x,y)= (y,x)$,  $i_2(x,y)=(-x,-y)$, and $i_3(x,y)=(x^{-1},y^{-1})$, and $G_0$ be the
normal subgroup generated by $i_2$ and  $i_3$. 

The quotient space $\PP^1 \times \PP^1  / G$ is isomorphic to $\PP^2$ and the quotient map 
$\pi : \PP^1 \times \PP^1  \to \PP^2$ is given in suitable affine coordinates by 
$(x,y) \mapsto (xy+\frac1{xy},\frac{x}{y}+ \frac{y}{x})$, and a fastidious computation shows that 
$\cW$ is the image of $[dx]$ under $\pi$.

Consider  the quotient space $X:= \PP^1 \times \PP^1  / G_0$.
Arguing as in  the previous lemma we see that $X$ is birationally equivalent to graph of $\cW$ hence that $\f$ lifts to $X$ and then to $\PP^1 \times \PP^1$ as an endomorphism. 
The claim follows.
\end{proof}

\section{Smooth quotients of abelian surfaces}\label{S:exam}

\subsection{Quotients of abelian surfaces isomorphic to $\mathbb P^2$} Table~\ref{T2} describes the $2$-tori that may appear in Case~(2) of Theorem~\ref{TI:C}. It is extracted from~\cite{KTY,toku-yoshida}, see also~\cite{dupont}.
The number $\tau$ is an arbitrary complex number with positive imaginary part, $L(\tau) = \Z + \tau \Z$, $i = \sqrt{-1}$, $\zeta = \exp(2i\pi/6)$, and
$G(m,p,2)$ is the reflection group of order $2m^2/p$ generated by the following three matrices:
$$G(m,p,2) =
\left\langle \left[\begin{array}{cc}
0 & 1\\ 1 & 0
 \end{array}\right]
, \left[\begin{array}{cc}
0 & e^{\frac{2i\pi}{m}}\\ e^{-\frac{2i\pi}{m}} & 0
 \end{array}\right]
,
\left[\begin{array}{cc}
  e^{\frac{2ip\pi}{m}} & 0 \\ 0& 1
 \end{array}\right] \right\rangle
~$$
In both cases, the last column refers to the possible integers $k\ge 2$ for which a map preserves an irreducible $k$-web.

\begin{table}[!h]
\begin{center}
\begin{tabular}{|c|c|c|c|}
\hline
Lattice & Group $G$ & $|G|$ &  Irreducible $k$-web \\
\hline
& & & \\
$L(\tau)\left(\begin{array}{c} 1\\0 \end{array}\right) +
 L(\tau)\left(\begin{array}{c} 0\\1 \end{array}\right)$
& $G(2,1,2)$ &  $8$ & $k \in \{  2, 4\} $\\
& & & \\
\hline
& & & \\
$L(i)\left(\begin{array}{c} 1\\0 \end{array}\right) +
 L(i)\left(\begin{array}{c} 0\\1 \end{array}\right)$
& $G(4,1,2)$ & $32$ & $  k \in  \{ 2, 4 , 8\} $\\
& & & \\
\hline
& & & \\
$L(\zeta)\left(\begin{array}{c} 1\\0 \end{array}\right) +
 L(\zeta)\left(\begin{array}{c} 0\\1 \end{array}\right)$
& $G(3,1,2)$ &  $18$ & $ k \in \{ 2,3,6 \} $\\
& & & \\
\hline
& & & \\
$L(\zeta)\left(\begin{array}{c} 1\\0 \end{array}\right) +
 L(\zeta)\left(\begin{array}{c} 0\\1 \end{array}\right)$
& $G(6,1,2)$ & $72$ & $k  \in \{ 2, 3, 6, 12\}$\\
& & & \\
\hline
& & & \\
$L(i)\left(\begin{array}{c} 1\\0 \end{array}\right) +
 L(i)\left(\begin{array}{c} 0\\1 \end{array}\right)$ &
$\langle G(4,2,2),
\frac{1+i}{2} \left(\begin{smallmatrix}
1\\ 1
 \end{smallmatrix}\right)\rangle$ &  $32$ & $ k \in \{ 2,  4 \} $\\
& & & \\
\hline
& & & \\
$L(\tau)\left(\begin{array}{c} 1\\-1 \end{array}\right) +
 L(\tau)\left(\begin{array}{c} \zeta\\\zeta^2 \end{array}\right)$
& $G(3,3,2)$ & $6$ & $k \in \{ 2, 3, 6 \}$\\
& & & \\
\hline
\end{tabular}
\end{center}
\caption{List of $2$-tori yielding an endomorphism of $\PP^2$ preserving a $k$-web, $k\ge 3$, see Theorem~\ref{TI:C}~item~(1).} \label{T2}
 \end{table}

\subsection{An example of an invariant $12$-web} We give an explicit example of an endomorphism of $\PP^2$ of degree $3$ preserving an irreducible $12$-web. Take the endomorphism $(x,y) \mapsto (\sqrt{3}\zeta^{-1} x, \sqrt{3}\zeta^{-1}y)$
on the $2$-torus $T\= (\C/\Z[\zeta])^2$. It preserves any linear foliation given by $dx + \la dy$ with $\la \in \C$.

We then take the quotient of $T$ by the subgroup $G_0$ of index $2$ of $G(6,1,2)$ generated by
$\left[\begin{smallmatrix}
  e^{\frac{i\pi}{3}} & 0 \\ 0& 1
 \end{smallmatrix}\right]$
and $\left[\begin{smallmatrix}
1 & 0 \\ 0& e^{\frac{i\pi}{3}}
 \end{smallmatrix}\right]$. The quotient is  isomorphic to $\mathbb P^1 \times \mathbb P^1$, and the isomorphism $T/G_0\simeq \PP^1 \times \mathbb P^1$ is given explicitly
  by the map $$(x,y)
\mapsto \left(
(\mathfrak{p}'(x))^2, (\mathfrak{p}'(y))^2\right) \text{ with } \mathfrak{p} = \frac1{x^2} + \sum_{\Z[\zeta]\setminus\{0\}} \frac{1}{(x-\tau)^2} - \frac1{\tau^2}$$ being the Weierstrass $\mathfrak{p}$-function. Using the fact that $\mathfrak{p}$  satisfies the differential equation $(\mathfrak{p}')^2 = 4\mathfrak{p}^3 -1$
 we obtain that the map induced on $\mathbb P^1 \times \mathbb P^1$ is equal to
$$
(x,y) \mapsto \left(\frac{-x}{3^3}\, \left( \frac{x+9}{x+1}\right)^2, \frac{-y}{3^3}\, \left( \frac{y+9}{y+1}\right)^2\right)
$$
and that it preserves the $6$-webs given by
$$
\frac{dx^6}{x^3(1+x)^4}+ \frac{\la\,dy^6}{y^3(1+y)^4}~.
$$
where $\lambda \in \mathbb C$ is arbitrary. This web is irreducible as soon as $\la^6 \neq 1$.

\smallskip

Taking then the quotient of $T$ by $G(6,1,2)$, we get that the map
\begin{multline*}
 (x,y) \mapsto
\left( - \, \frac{y^2x + 14\,y^2 +126\, y+113\, xy+81x +18\,x^2 +x^3 +2 yx^2}{3^3(x+y+1)^2}, \right. \\
\left. \frac{y}{3^6} \, \left( \frac{9x+y+9^2}{x+y+1}\right)^2 \right)
\end{multline*}
preserves an irreducible $12$-web. 
It is possible to write down explicit formulas for the  invariant irreducible $12$-webs
with the help of a computer algebra system but the expressions we obtained 
are too messy to be included here.

\end{document}